\documentclass{amsart}
\usepackage{amsfonts}
\usepackage{amsmath}
\usepackage{amssymb}
\usepackage{graphicx}
\usepackage{version}
\usepackage{caption}%
\newcommand{\R}{{\mathbb R}}
\theoremstyle{plain}
\newtheorem{theorem}{Theorem}[section]
\newtheorem{cor}{Corollary}[section]
\newtheorem{lemma}{Lemma}[section]
\newtheorem{proposition}{Proposition}[section]
\theoremstyle{definition}
\newtheorem{definition}{Definition}[section]
\newtheorem{example}{Example}[section]
\newtheorem{remark}{Remark}[section]
\newtheorem{conjecture}{Conjecture}[section]

\numberwithin{equation}{section}
\excludeversion{comment1}
\begin{document}

\title[Thresholds of IFS Families] {Thresholds for One-Parameter Families of Affine Iterated Function Systems}
\author[A. Vince]{Andrew Vince \\ University of Florida \\ Gainesville, FL, USA}
\address{Department of Mathematics \\ University of Florida \\ USA}
\email{\tt  avince@ufl.edu} 
\thanks{This work was partially supported by a grant from the Simons Foundation (322515 to Andrew Vince).}

\begin{abstract}  This  paper examines thresholds for certain properties of the attractor of a general
one-parameter affine family of iterated functions systems.  As the parameter increases, the iterated function system becomes 
less contractive, and the attractor evolves.  Thresholds are studied for the following properties:  the existence of
an attractor, the connectivity of the attractor, and the existence of non-empty interior of the attractor.  Also
discussed are transition phenomena between existence and non-existence of an attractor.  
\end{abstract}

\maketitle

\section{Introduction}

An iterated function system (IFS) in this paper is a finite set 
\[ F=\{f_{1},f_{2}, \dots ,f_{N}\}\]
of $N\geq 2$ affine functions from $\R^d$ to $\R^d$.   A function $f$ is {\it affine} if it is of the form $f(x) = Lx+a$,
where $L$ is a {\it non-singular} linear map ($n\times n$ real matrix) and $a\in \R^d$.  The linear map $L$ will be referred to as the
{\it linear part} of $f$, and $a$ as the {\it translational part} of $f$.  A special case of an IFS
is a {\it similarity IFS} for which each function is a similarity transformation. 

For the collection $\mathbb{H}$ of non-empty compact subsets of
$\R^d$, the classical Hutchinson operator $F\,:\,{\mathbb{H}}(\R^d)\rightarrow
{\mathbb{H}}(\R^d)$ is given by 
\[ F(K)=\bigcup_{f \in F} f (K).\]
By abuse of language, the same notation $F$ is used for the IFS, the set of functions in the IFS, and for the
Hutchinson operator; the meaning should be clear by the context.  A compact set $A$ is the {\it attractor} of $F$ if
\begin{equation} \label{eq:lim}  A=\lim_{k\rightarrow\infty}F^{k}(K),\end{equation}
where $F^k$ denotes the $k$-fold composition; the limit is with respect to the Hausdorff metric and is independent of the set $K\in \mathbb H$.
Note that, if it exists, the attractor is the unique compact set $A$ such that
\begin{equation} \label{eq:fixed} 
A=F(A).
\end{equation}

One-parameter affine IFS families are the topic of this paper.  Let 
 $F = \{f_1, f_2, \dots, f_N\}$ be a set of non-singular affine transformations on $\R^d$  and 
$Q = \{q_1, q_2, \dots, q_N\}$  a set of vectors in $\R^d$.  We consider the general family of affine IFSs
 \begin{equation} \label{eq:f1} F_t = \{  t\, f_i(x) + q_i \, : 1\leq i \leq N \}\end{equation}
depending on a real parameter $t\geq 0$.  Call such a  family a {\bf one-parameter affine family}.
A one-parameter family is called a {\bf similarity family} is each function is a similarity transformation. 
 The parameter $t$ is directly related to the contractivity of the IFS, the nearer the parameter $t$ is to $0$, the more contractive 
the functions in the IFS.   Example~\ref{ex:intro} below shows the evolution of the attractor $A_t$ of a particular one-parameter similarity 
family $F_t$.  

The goal of this paper is to examine thresholds for certain properties of the attractor 
$A_t$ of $F_t$.  In particular, we ask whether there exist numbers $t_0, t_1, t_2$ such that:

\begin{itemize}
\item $F_t$ has an attractor for $t < t_0$, but no attractor for $t>t_0$.  
\item $A_t$ is disconnected for $t<t_1$, but connected for $t>t_1$.
\item $A_t$ has empty interior for  $t<t_2$, but non-empty interior for $t>t_2$.
\end{itemize}
We also examine phenomena at the point $t_0$ of transition between attractor and no attractor. 
 
 An attractor with non-empty interior is important in certain IFS constructions of tilings of Euclidean space (see \cite{BV2,BV3}) .  
Note that, in general, there is no direct relation between whether the attractor of an IFS is connected and whether it has non-empty interior.  
For example, the Koch curve is connected, but has empty interior, and the attractor in Example~\ref{ex:shooter} (Figure~\ref{fig:shooter}) is disconnected, but 
has non-empty interior.

\begin{example}  \label{ex:intro}  As an example on $\R^2$, consider the one-parameter similarity family
$F_t = \{ f_t , \;g_t  \},$
where 
\[ f_t \begin{pmatrix} x\\ y \end{pmatrix} = t\, Q \begin{pmatrix} x\\ y \end{pmatrix}, \qquad \qquad 
g_t \begin{pmatrix} x\\ y \end{pmatrix} =  t \Bigg (.4\, Q \begin{pmatrix} x\\ y \end{pmatrix} - \frac{.4}{\sqrt{2}}  
\begin{pmatrix} 1\\ 1 \end{pmatrix} \Bigg )+ \begin{pmatrix} 1\\ 0 \end{pmatrix}\Bigg \}\]
and $Q$ is the rotation by $\pi/4$: \[Q = \frac{1}{\sqrt{2}}\, \begin{pmatrix} 1 &-1 \\ 1 & 1 \end{pmatrix}.\]
 Figure~\ref{fig:morph} shows the attractor of $F_t$ for increasing values of $t$. 

\begin{figure}[htb]  
\vskip 3mm
\includegraphics[width=5cm, keepaspectratio]{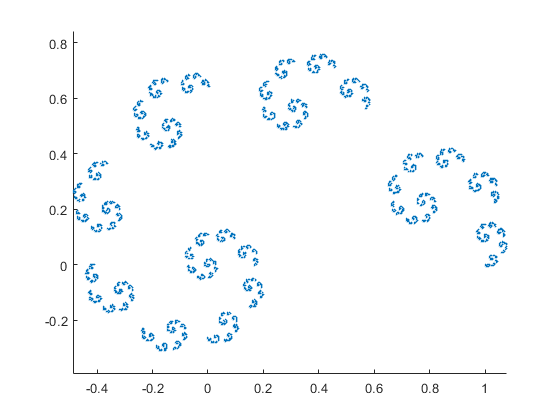}\includegraphics[width=5cm, keepaspectratio]{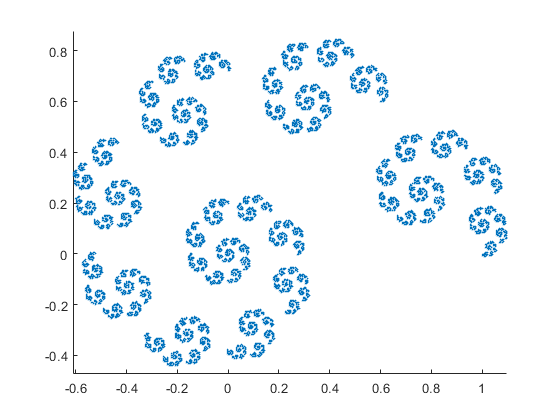}\includegraphics[width=5cm, keepaspectratio]{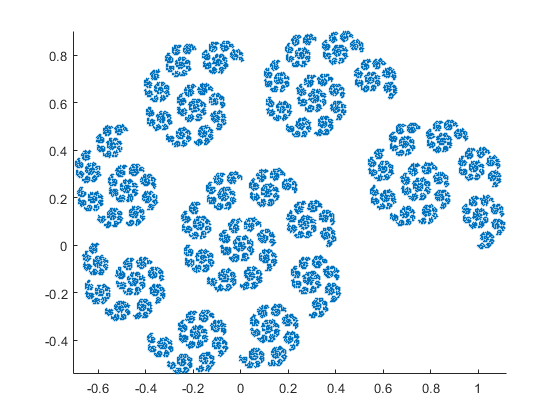} \vskip 5mm \includegraphics[width=5cm, keepaspectratio]{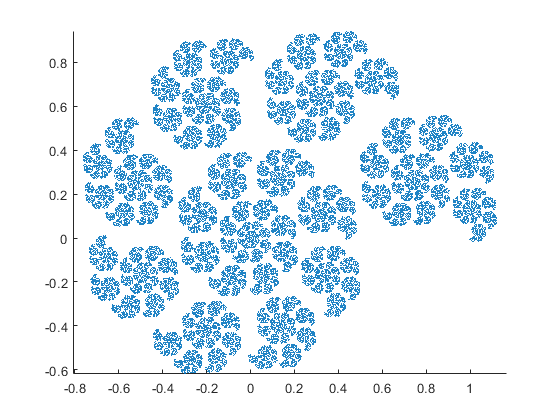}\includegraphics[width=5cm, keepaspectratio]{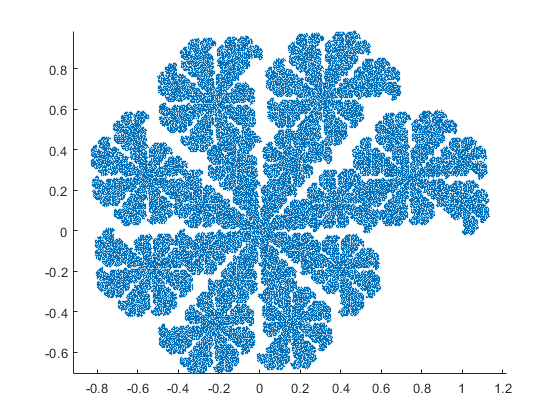}\includegraphics[width=5cm, keepaspectratio]{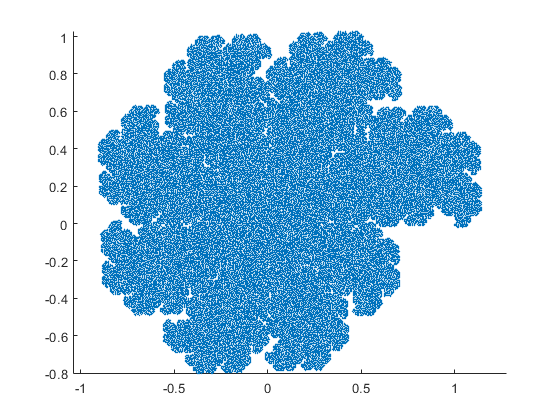}
\caption{The attractor $A_t$ for the one-parameter affine family $F_t$ of  Example~\ref{ex:intro} for successive parameter values $t = .8, .85, .88, .9, .92, .94$.}
\label{fig:morph}
\end{figure}
\end{example}

\section{Previous  Results} \label{sec:previous}

A particular $2$-dimensional one-parameter similarity family that has emerged as a topic of interest \cite{B1,B2,MB,S1,S2} is
\begin{equation*} \label{eq:M} \{ t Q(x), \, t Q(x) +(1,0)\},\end{equation*}
where $Q$ is a rotation about the origin.  This family is often expressed in  complex form, with
complex parameter $\tau$, as
\begin{equation} \label{eq:man} F_{\tau}= \{ \tau z, \, \tau z+1\},\end{equation}
 where $\tau \in \mathbb D = \{z : |z| < 1\}$. 
 It was shown already in \cite{DK} that, for the family \eqref{eq:man}, if $\tau \notin \R$ and $|\tau|$ is sufficiently close to $1$, then $A_{\tau}$ has non-empty interior. 

If $A_{\tau}$ denotes the attractor of $F_{\tau}$, then the set
\[M := \{ \tau \in \mathbb D : A_{\tau} \; \text{ is connected}\},\]
introduced in \cite{BH} as an analog of the classical Mandelbrot set, is called the Mandelbrot set for $F_{\tau}$.
 The portion of $M$ in the first quadrant is shown in white in Figure~\ref{fig:mandel}, this graphic due to Christoph Bandt.  It is shown in \cite{Bo}
that $M$ is connected and locally connected. 

\begin{figure}[htb]  
\includegraphics[width=7cm, keepaspectratio]{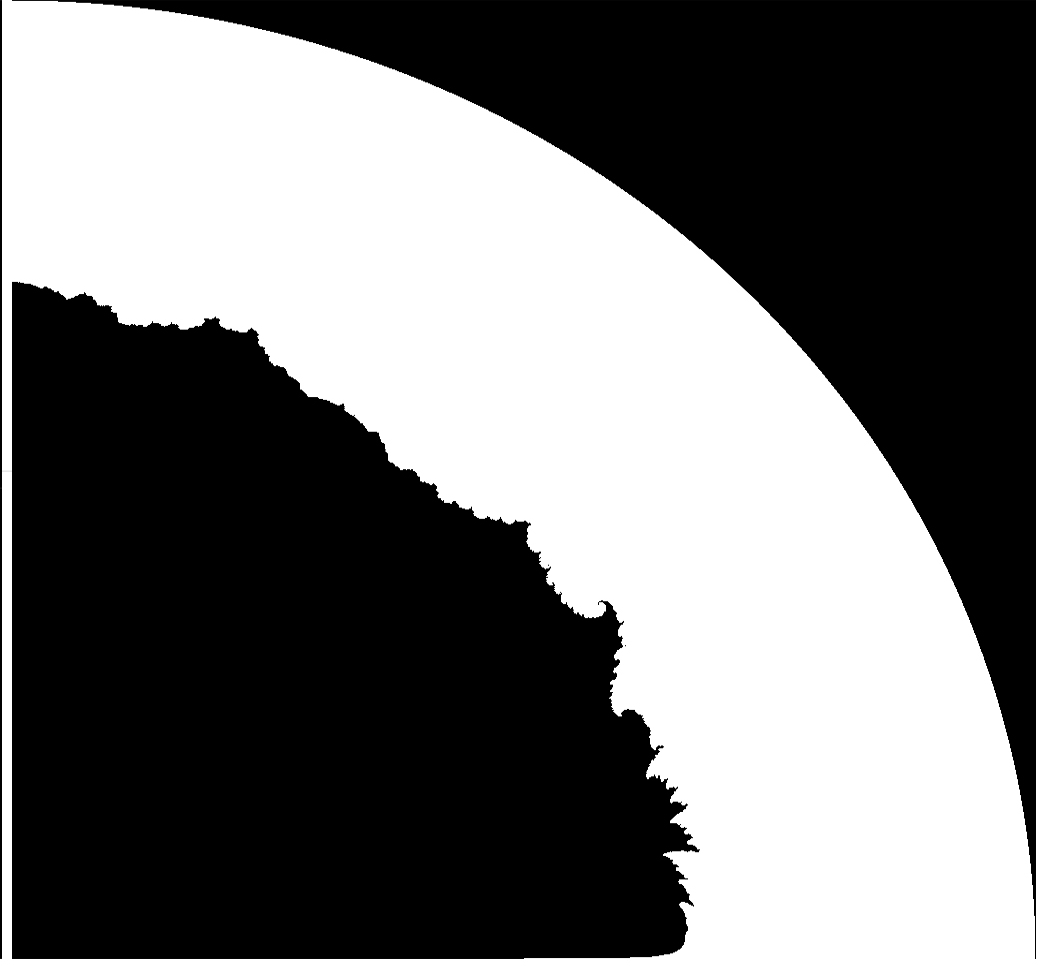}
\caption{Mandelbrot set for the family \eqref{eq:man}, due to C. Bandt.}
\label{fig:mandel}
\end{figure}

 Call an IFS $F$ in $\R^d$ {\it degenerate} if there exists an invariant affine subspace of dimension less than $d$ that is common to all the functions in $F$; otherwise call $F$ {\it non-degenerate}.  Hare and Siderov \cite{HS1, HS2} studied the IFS $F_M := \{Mv -u, Mv+u\}$, where $M$ is a $d\times d$ real matrix and $u\in \R^d$ is a vector such that $\text{span}\{M^n u: n\geq 0\} = \R^d$.  They proved that if the IFS is non-degenerate and if $|\det M| \geq 1/2$, then the attractor $A_{F}$ is connected, and if $|\det M| \geq 2^{-1/d}$,
 then $A_F$ has non-empty interior.  The latter occurs if the modulus of each eigenvalue of $M$ is between  $2^{-1/{d^2}}$ and $1$.  If the
dimension $d=2$ and the eigenvalues of $M$ are complex, then, via a conjugacy, the IFS $F_M$ is
equivalent to the IFS $F_{\tau}$ of \eqref{eq:man}.  In this case, their result states that if $1 > |\tau| > 1/2$, then $A_{\tau}$ is connected, and if $1 > |\tau| > 2^{-1/4} \approx 0.84$, then $A_{\tau}$ has non-empty interior. 

 In \cite{LW,LY} the authors study connectedness and disk-likeness of the attractor for a particular family of self-similar {\it digit tiles} (see \cite{V}) in $\R^2$.  The family depends on a parameter involving the digits.  They provide sufficient conditions on the parameter for the attractor to be disk-like and
necessary and sufficient conditions for the attractor to be connected.  

\section{Organization and Summary of Results} \label{sec:organize}

Concerning a threshold for the existence of an attractor of the one-parameter affine family $F_t$ of Equation~\eqref{eq:f1}, a complete answer is provided in 
Section~\ref{sec:existence}.  
\begin{itemize}
\item
There is a threshold $t_0 = 1/\rho(F)$, where $\rho(F)$ is the joint spectral radius of the linear parts
of the functions in $F$.  For all $t$ such that $t\in [0,t_0)$, the affine IFS $F_t$ has a unique attractor, and for all $t >t_0$ the $F_t$ has no attractor (see Corollary~\ref{cor:existence}).  
\end{itemize}
The definition of  joint spectral radius is given in Section~\ref{sec:existence}.  

At the point $t_0$ between the 
existence and non-existence of an attractor, certain transition phenomena can occur.  When $F_t$ is what we call a {\it bounded} family (in particular, for a bounded family there is a ball $B$ such that  $A_t \subset B$ for $t\in [0,t_0)$),  
{\it transition attractors} are defined in Section~\ref{sec:transition}.  The {\it lower transition attractor} $A_*$ is unique. The main conjecture of this paper
  (Conjecture~\ref{conj:lim})  is that, for a bounded family $F_t$, there is a  unique {\it upper transition attractor} rather than multiple upper transition attractors.  
If this is the case, then the upper transition attractor is $\lim_{t\rightarrow t_0} A_t$.  The conjecture  is true in the
$1$-dimensional case, and graphical evidence seems to support it in dimension $2$.    
The following results, for lower transition attractor $A_*$ and any upper transition attractor $A^*$, appear in Section~\ref{sec:transition}.  
\begin{itemize}
\item  Although $A_*$ and $A^*$ may not satisfy the Equation~\eqref{eq:lim} that defines the attractor of an IFS, 
they do satisfy the fixed point property of Equation~\eqref{eq:fixed}, i.e.  $F_{t_0}(A^*) = A^*$ and $F_{t_0}(A^*) = A^*$.
\item It is not necessarily the case that $A_* = A^*$, but  $A_*\subseteq A^*$ and $conv (A_*) = conv (A^*)$, where {\it conv} denotes the convex hull.
\end{itemize} 

Some of our results hold for what we call {\it semi-linear} one-parameter families;  some hold with the exception of {\it linear} or {\it quasi-linear} one-parameter families. The definitions and properties of these types of families  appear in Section~\ref{sec:linear}.
\vskip 2mm

Concerning a threshold for the connectivity of the attractor of a one-parameter affine family $F_t$, the following
facts are the subject of Section~\ref{sec:connected}.
\begin{itemize}
\item If the attractor $A$ of an affine IFS $F$ is not connected, then $A$ must have infinitely many components.  If there are exactly two functions in $F$ and $A$ is not connected,
then $A$ must be  totally disconnected; this, however, is not true in general for an IFS with more than two functions.
\item The attractor $A_t$ of a one-parameter affine family may be disconnected for all $t \in [0,t_0)$. 
\item  However, for a   one-parameter similarity family $F_t$, there exists a real number $\overline{t_1} > 0$ 
such that $A_t$ is connected for all $t \in ( \overline{t_1}, t_0)$ (Theorem~\ref{thm:sim}).
\item The attractor  $A_t$ of any one-parameter quasi-linear family $F_t$ is connected for all $t \in [0,t_0)$.
\item However, for any affine, but not quasi-linear, one-parameter family $F_t$, there exists a real number $\widehat{t_1} >0$ 
such that $A_t$ is disconnected for all $t\in [0,\widehat{t_1})$ (Theorem~\ref{thm:aff2}).
\item Even if $F_t$ is a one-parameter similarity family, there may be no single threshold $t_1$ such that $A_t$ is disconnected for all $t \in (0,t_1)$ and is connected for all $t\in (t_1,t_0)$.
\end{itemize}

In Section~\ref{sec:connected}, however, a slightly weaker notion of connectivity, which we call {\it weak connectivity} (Definition~\ref{def:wc}), is introduced that satisfies 
the following stronger threshold property (Theorem~\ref{thm:wc}):
\begin{itemize}
\item  For a  semi-linear, but not linear, one-parameter affine family $F_t$,  there is a threshold $t_1>0$ such that $A_t$ is weakly connected 
for all $t \in (t_1, t_0)$ and 
strongly disconnected for all $t\in [0,t_1)$.
\end{itemize} 

A threshold for the appearance of non-empty interior in the attractor $A_t$ of  a one-parameter affine family $F_t$
is the subject of Section~\ref{sec:interior}.  The following facts are proved in that section.
\begin{itemize}
\item For a degenerate one-parameter affine family $F_t$, the attractor trivially has empty interior for all $t\in [0,t_0)$.  This is even the case for some
non-degenerate similarity families (Example~\ref{ex:allempty} and Theorem~\ref{thm:empty2}).
\item  For a one-parameter affine family $F_t$, there is a real number $\tau >0$ such that $A_t$ has empty interior for all 
$t\in (0,\tau)$ (Theorem~\ref{thm:empty}).
\item For  a  tame (Definition~\ref{def:tame}) and semi-linear similarity family $F_t$,
 there is a real number $t_2 > 0$ such that  $A_t$  has empty interior for $t\in [0,t_2)$ and non-empty interior for 
$t\in (t_2,t_0 )$.  (We have no example of a semi-linear similarity IFS family that is not tame.)
\item Even when $F_t$ is not tame, Theorem~\ref{thm:nonempty} and Corollary~\ref{cor:nonempty} provide 
sufficient conditions under which there exists $\tau <t_0$ such that  $A_t$ has non-empty for 
$t\in (\tau,t_0 )$.  
\end{itemize}

\section{Threshold $t_0$ for the Existence of an Attractor} \label{sec:existence}

Concerning a threshold for the existence of an attractor of the one-parameter affine family, Corollary~\ref{cor:existence} below, gives a complete answer. 
The terminology is as follows.  An IFS $F$ is {\it contractive} if $f$ is a 
contraction for all $f \in F$
with respect to some metric that is  Lipschitz equivalent to the standard Euclidean metric.  An IFS $F$ is 
{\it topologically attractive} if there exists a compact set $K \in \mathbb H$ such that  $F(K) \subset K^o$, where
$K^o$ denotes the interior of $K$.  A {\it convex body} is a convex set with non-empty interior.  

The {\it joint 
spectral radius} $\rho(F)$ of an IFS $F$ is the joint spectral radius of the set of linear parts of the functions in $F$. The joint 
spectral radius of a set ${\mathbb{L}}=\{L_{i},\,i\in I\}$ of linear
maps was introduced by Rota and Strang \cite{RS} and the generalized spectral
radius by Daubechies and Lagarias \cite{DL}. Berger and Wang \cite{BW} proved
that the two concepts coincide for, in particular, a finite set of maps.  What follows is the definition of
the joint spectral radius of ${\mathbb{L}}=\{L_{1}, L_2, \dots, L_N\}$.  Let $\Omega_{k} $ be the set of
all words $i_{1}\,i_{2}\,\cdots\,i_{k}$, of length $k$, where $i_{j}\in
\{1,2,\dots, N\}\,\,1\leq j\leq k$. For $\sigma=i_{1}\,i_{2}\,\cdots\,i_{k}\in\Omega_{k}$,
define
\[
L_{\sigma}:=L_{i_{1}}\circ L_{i_{2}}\circ\cdots\circ L_{i_{k}}.
\]
For a linear map $L$, let $\rho(L)$ denote the ordinary spectral radius, i.e.,
the maximum of the moduli of the eigenvalues of $L$. 
The {\it joint
spectral radius} of ${\mathbb{L}}$ is
\[
\hat{\rho}=\hat{\rho}({\mathbb{L}}):=\limsup_{k\rightarrow\infty}\hat{\rho
}_{k}^{1/k}\qquad\mbox {where}\qquad\hat{\rho}_{k}:=\sup_{\sigma\in\Omega_{k}}\,\|L_{\sigma}\|,
\]
which does not depend on the chosen matrix norm.  
The {\it generalized spectral radius} of ${\mathbb{L}}$ is
\[
\rho=\rho({\mathbb{L}}):=\limsup_{k\rightarrow\infty}\rho_{k}^{1/k}%
\qquad\mbox {where}\qquad\rho_{k}:=\sup_{\sigma\in\Omega_{k}}\,\rho(L_{\sigma
}).
\] 
Our Theorem~\ref{thm:jsr} below generalizes and extends a fundamental 
result of Hutchinson \cite{hutchinson}.  
 
\begin{theorem}[\cite{BV1} Theorem 4] \label{thm:jsr} The following are equivalent for an affine IFS $F$.
\begin{enumerate}
\item $F$ has an attractor.
\item $F$ is contractive.
\item $F$ is topologically attractive with respect some convex body $K$.
\item The joint spectral radius $\rho(F) < 1$.
\end{enumerate}
\end{theorem}

\begin{cor} \label{cor:existence} Given a one-parameter affine family $F_t$, an attractor $A_t$ exists for $t < 1/\rho(F)$ and fails to exist for $t \geq 1/\rho(F)$.   In other words, the threshold for the existence of the attractor is $t_0 = 1/\rho(F)$.
\end{cor}

\begin{proof} If the linear parts of the affine functions in $F$ are $\{B_1, B_2, \dots, B_N\}$, then by 
Theorem~\ref{thm:jsr}
an attractor of $F_t$ exists if $t\rho(F) =  t \rho(B_1, B_2, \dots, B_N) =   \rho(tB_1, tB_2, \dots, tB_N)  =
\rho(F_t) <1$,
 i.e., if $t < 1/\rho(F)$.  And an attract fails to exist if $F_t$ exists if $t\rho(F) =  t \rho(B_1, B_2, \dots, B_N) =   \rho(tB_1, tB_2, \dots, tB_N)  =
\rho(F_t) = >1$, i.e., if $t > 1/\rho(F)$.
\end{proof}  

It is known that the joint spectral radius is NP-hard to compute or to approximate.   Moreover, the question of whether $\rho < 1$  is an undecidable problem.  
However, for an IFS consisting of similarity transformations, the joint 
spectral radius is easily verified to be the maximum of the scaling ratios of the set of similarities in the IFS $F$.  

\begin{cor} \label{cor:existence2}  Let $F_t = \{  t\, f_i(x) + q_i \, : 1\leq i \leq N \}$ be a one-parameter similarity family and  
$F = \{f_1, f_2, \dots, f_N\}$.  If the scaling ratio of the similarities in $F$ are
$r_1, r_2, \dots, r_N$, the threshold for the existence of an attractor of $F_t$ is  $t _0 = 1/( \max_{1\leq i \leq N} r_i)$.
\end{cor}

The following theorem is relevant to results in Section~\ref{sec:transition}.

\begin{theorem}  \label{thm:continuity}  For the one-parameter affine family $F_t$, the function $t\mapsto A_t$ is continuous on the interval $(0,t_0)$. 
\end{theorem}

\begin{proof}   By Theorem~\ref{thm:jsr}, all the functions in $F$ are contractions with respect to a metric 
$d(\cdot, \cdot)$ that is Lipschitz equivalent to the standard Euclidean metric.  So there are constants
$c_1>0, \, c_2 >0$ such that $c_1 |x-y| \leq d(x,y) \leq c_2|x-y|$ for all $x,y \in \R^d$.  Restrict our  metric to the
convex body  $K$ as defined in Theorem~\ref{thm:jsr}.  According to \cite[Theorem 11.1]{MB}, it is sufficient to show that
$d(f_t(x), f_{t'}(x)) \leq c |t-t'|$ for some constant $c$ independent of $t, t', x\in K$, and $f\in F$.
Let $c_3$ be a constant such that $|f(x)| \leq c_3$ for all $x\in K$ and all $f\in F$.  Then
\[d(f_t(x), f_{t'}(x)) \leq c_2 |f_t(x) - f_{t'}(x)| = c_2|(t-t')| \, | f(x)|  \leq c_2 c_3 |t-t'|.\]
\end{proof}

\begin{example}[$t\mapsto A_t$ is Continuous on the Interval $(0,t_0)$ but not Uniformly Continuous]  \label{ex:I} 
In general, the continuity guaranteed by Theorem~\ref{thm:continuity} is not uniform continuity.  For example, for the $1$-dimensional 
one-parameter family
$F_t = \{ -tx+t+1, \, -tx -t -1\}$ the attractor is the closed interval  \[A_t = \Bigg [ -\frac{1+t}{1-t}, \frac{1+t}{1-t} \Bigg ],\]  
whose length goes to infinity as $t$ approaches $t_0=1$.  Consequently continuity is not uniform.
This question of when continuity is uniform arises in Section~\ref{sec:transition}.  
\end{example}

\begin{remark} \label{rem:noloss}
Given an IFS family $F_t$, let $F'_t = t_0\, F_t$.  The one-parameter affine family $F'_t$, as $t$ varies between $0$ and $1$, is the same as the family $F_t$ as $t$ varies
between $0$ and $t_0$.  Therefore there is no loss of generality in assuming that $t_0 = 1$.
\end{remark}

\section{Linear, Quasi-Linear, and Semi-Linear Families} \label{sec:linear}

\begin{proposition}  \label{prop:fixed}  Let $F_t$ be a one-parameter affine  family.  If $t\in [0,t_0)$, then every function in $F_t$ has a unique fixed point, and that fixed point is contained in the attractor $A_t$.  
\end{proposition}

\begin{proof}   Let $f_t = t f(x) +q \in F_t$.  A point $x$ is a fixed
point of $f_t$ if and only if $(\frac{1}{t}-L)x = a +\frac{q}{t}$, where $L$ is the linear part of $f$ and $a$ is the translational part.  Therefore $f_t$ has a unique 
fixed point unless $1/t$ is an eigenvalue of $L$.  If  $1/t$ is an eigenvalue of $L$, however, then by the definition of the generalized spectral radius \[ 1/t_0 = \rho(F) \geq 1/t,\]
which implies $t \geq t_0$, a contradiction. 

If $x$ is a fixed point of a function $f$ of an IFS with attractor $A$, then $x = \lim_{n\rightarrow \infty} f^n(x) \in A$.
\end{proof}

Two functions $f, g: \R^d \rightarrow \R^d$ are {\it conjugate by a function} $h$ if $g = h^{-1}\, f\, h$.  

\begin{definition}\label{def:items}  Let $F_t  = \{ f_{(i,t)}(x)  := t\, f_i(x) + q_i \, : 1\leq i \leq N \}$ be a one-parameter affine family. 
\begin{itemize}
\item  Call $F_t$ {\bf linear} if, for all $i = 1,2, \dots, N$ and for all  $t\in [0,t_0)$, the function $f_{(i,t)}$ is conjugate to a linear function by a by a function $h$ that is independent of $i$ and $t$.  
\item Call  $F_t$ {\bf quasi-linear}  if, for all $i = 1,2, \dots, N$ and for all  $t\in [0,t_0)$, the function $f_{(i,t)}$ is conjugate  to a linear function by a by a function $h_t$ that is independent of $i$.    
\item Call $F_t$ {\bf semi-linear}  if,  for all $i = 1,2, \dots, N$ and for all  $t\in [0,t_0)$, the function $f_{(i,t)}$ is conjugate  to a linear function by a by a function $h_i$ 
that is independent of $t$. 
\end{itemize}
\end{definition}

It follows from the definitions that a linear family is necessarily quasi-linear and semi-linear.  Converse statement do not hold.   By Propositions~\ref{prop:linear} and \ref{prop:quasi} below, the attractor $A_t$ for a linear or quasi-linear family $F_t$ is a single point for all  $t\in [0,t_0)$. This is not the case for a semi-linear family.  
 The following two lemmas are easily verified.

\begin{lemma}  \label{prop:props} Let $f : \R^d\rightarrow \R^d$ be an affine function, $q \in \R^d$, and $t> 0$.    The following statements are equivalent for 
$g(x) = t\, f(x) +q$:
\begin{enumerate}
\item $q$ is a fixed point of $g$; 
\item  $f(q) = 0$; 
\item  $g(x) =t\, L(x-q)+q$, where $L$ is the linear part of $f$. 
\end{enumerate}
\end{lemma}

\begin{lemma} \label{lem:trans} For an affine function $f = g^{-1}\, L_1\, g$, where $g(x) = L_2(x) + a$ and $L_1$ and $L_2$ are linear, there is
a linear function $L_3$ such that $f = h^{-1}\, L_1\, h$, where $h(x) = x+a$.
\end{lemma}

\begin{proposition}  \label{prop:linear} The following statements are equivalent for a one-parameter affine family $F_t  = \{ f_{(i,t)}(x)  := t\, f_i(x) + q_i \, : 1\leq i \leq N \}$.  
\begin{enumerate}
\item   $F_t$ is linear.
\item   There is a point $q$ such that $q=q_i$ for all $i=1,2,\dots, N$, and $q$ is the unique fixed point of $f_{(i,t)}$  for all $i$ and all $t \in [0,t_0)$.  
\end{enumerate}
Moreover, for a linear family the attractor $A_t$ of $F_t$ is the single point $q$.  
\end{proposition}

\begin{proof}  (1) $\Rightarrow$ (2) Assume that  $L = h^{-1}\, f_{(i,t)} \, h$, where $L$ is linear.  By Lemma~\ref{lem:trans}, we can assume, without loss
of generality, that $h$ is a translation.  By Proposition~\ref{prop:fixed}, each function in $F_t, \, t\in [0,t_0),$ has a unique fixed point; call it $q_{(i,t)}$.  The unique fixed point of the linear map $h^{-1}\, f_{(i,t)} \, h$ is the origin; therefore $h(x) = x+q_{(i,t)}$.  By the definition of linear, $q_{(i,t)}$ is independent of $i$ and $t$; call it $q$.
Now $q = t\, f_i(q)+ q_i$ for all $t\in [0,t_0)$, which implies that $f_i(q)=0$ and hence $q = q_i$ for all $i$.   

(2) $\Rightarrow$ (1)  Let $h(x) = x+q$, and let $L_i$ denote the linear part of $f_i$.  Then 
 $(h^{-1}\, f_{(i,t)} \, h)(x) = (h^{-1}\, f_{(i,t)})(x+q) =
 h^{-1}( t\, L_i((x+q)-q)+q) =  h^{-1}( t\, L_i(x)+q) = t\, L_i(x)$.

The last statement in the proposition follows from the fact that the attractor of an IFS, all of whose functions are linear, is a single point, the origin.   
\end{proof}

\begin{proposition} \label{prop:quasi}  The following statements are equivalent for a one-parameter affine family 
$F_t  = \{  f_{(i,t)}(x)  := t\, f_i(x) + q_i \, : 1\leq i \leq N \}$. 
\begin{enumerate}
\item $F_t$ is quasi-linear.
\item  There are points $q_t, \, t\in [0,t_0),$  such that $q_t$ is the unique fixed point of $f_{(i,t)}$ for all $i=1,2,\dots, N$.  
\item The attractor $A_t$ of $F_t$ is a single point for all $t\in [0,t_0)$.  
\end{enumerate}
\end{proposition}

\begin{proof} The equivalence of the first two statements, as well as (1) $\Rightarrow$ (3), can be shown exactly as was done in the proof of Proposition~\ref{prop:linear}.  

(3) $\Rightarrow$ (1) Assume that $F_t$ is not quasi-linear.  Then, for some $t\in[0,t_0)$ and some $i\neq j$, the fixed points of  $f_{(i,t)}$ and $f_{(j,t)}$ must be different.  By Proposition~\ref{prop:fixed}, the attractor $A_t$ of $F_t$ contains all fixed points of the functions in $F_t$.  
Therefore, the attractor $A_t$ must contain at least two points.  
\end{proof}

\begin{example}[A Quasi-Linear Family that is not Linear]
\label{ex:surprise}
On $\R^2$, let
\[F_t = \Bigg \{f_t  \begin{pmatrix} x \\ y \end{pmatrix} = t\, \begin{pmatrix} x  \\  y+1 \end{pmatrix} + \begin{pmatrix} 1\\ 0 \end{pmatrix}, \; g_t    \begin{pmatrix} x \\ y \end{pmatrix} = t\, \begin{pmatrix} y+1 \\ -x+2y+2 \end{pmatrix} +\begin{pmatrix} 1 \\  0 \end{pmatrix} \Bigg \}. \] 
By Proposition~\ref{prop:linear}, $F_t$ is not a linear family because $q= (1,0)$ is not a fixed point of $f_t$ (nor of $g_t$).    Nevertheless, a straightforward calculation shows that $\Big (\frac{1}{1-t}, \,\frac{t}{1-t} \Big )$ is a fixed point of both $f_t$ and $g_t$ for all $t\in [0,t_0) = [0,1)$.  
\end{example}

In fact, examples of quasi-linear families that are not linear, like the one in Example~\ref{ex:surprise}, can be constructed, in general, as follows.  Let $W$ be any proper subspace of $\R^n$.  Let $L_i, \,i=1,2,\dots ,N,$ be linear maps that are equal when restricted to $W$, i.e.,  $L_i|_{W} = L_j|_{W}$ for all $i,j$, but $L_i \neq L_j$ for at least one pair $i\neq j$.  Let $w_0$ be any point in $W$ and $a$ any point not in $W$.  Let $q$ be any point 
such that $q-a \in W$.  Finally, let $f_i(x) = L_i(x-a) +w_0$ and $F_t = \{ t f_i(x) + q : i = 1,2,\dots, N\}$.  To show that, for a given $t$, all the functions in $F_t$ have the same fixed point, first notice that, for all $w\in W$ and all $i$, we have
\[t f_i(a+w) + q = t (L_i(a+w-a) + w_0) + q = t L_i(w) + t w_0 + a + w_1 = a + ( t L_i(w) + w_2) \in a + W,\]
for some $w_1 \in W$ and $w_2 = t w_0 + w_1$.  Since $L_i(w) = L_j(w)$ for all $i,j$, also  $t f_i(a+w) + q = t f_j(a+w) + q$ for all
$i,j$ and for all $t$.  Using the notation $f_{i,t}(x) = t f_i(x)+q$, we have that the affine subspace $a+W$ is invariant under $f_{i,t}$
and that $f_{i,t}|_{a+W} = f_{j,t}|_{a+W}$ for all $i,j$.   Therefore for $t<t_0$, all the functions $f_{i,t}$ in $F_t$ have the
same fixed point on the affine subspace $a+W$.

\begin{proposition} \label{prop:semi}  The following statements are equivalent for a one-parameter family 
$F_t  = \{ f_{(i,t)}(x) := t\, f_i(x) + q_i \, : 1\leq i \leq N \}$. 
\begin{enumerate}
\item $F_t$ is semi-linear.
\item  The point $q_i, \, i=1,2,\dots N,$ is the unique fixed point of $f_{(i,t)}$ for all $t\in [0,t_0)$.   
\end{enumerate}
Moreover, $q_i$ is also a fixed point of $f_{(i,t_0)}$ for all $i$.  
\end{proposition}

\begin{proof}  The proof of the equivalence of statements (1) and (2) is the same as for Proposition~\ref{prop:linear}.
That  $q_i$ is also a fixed point of $f_{(i,t_0)}$ follows from Lemma~\ref{prop:props};  $f_i(q_i)=0$ implies $f_{(i,t_0)}(q) := t_0\, f_i(q) + q_i = q_i$.
\end{proof}

\begin{example}[A Semi-Linear Family that is not Quasi-Linear]  There are many such examples. For example, 
The one-parameter similarity family of Example~\ref{ex:intro} is semi-linear, but not quasi-linear.  
\end{example}

\section{Threshold for Connectivity of the Attractor} \label{sec:connected}

\begin{theorem}  The attractor of an affine IFS cannot have a finite number, greater than one, of connected components.
\end{theorem}

\begin{proof}  Let $K$ be a connected compact set such that $F(K) \subset K^o$ as guaranteed by Theorem~\ref{thm:jsr}.  If $F^n(K)$ is connected for all $n$, then so is $A_F$, being the intersection
of compact connected sets.  Otherwise there is a first $m$ such that $F^m(K)$ is not connected.  By substituting 
$F^m(K)$ for $K$, it can be assume without loss of generality that $m=1$.  Now assume that, again without loss of
generality, that $f_1(K)$ and $f_2(K)$ lie in distinct components of $F(K)$.  This implies that 
$f_1\circ f_1(K), \,  f_1\circ f_2 (K), \, f_2\circ f_1(K), \, f_2\circ f_2(K)$ all lie in four distinct components of 
$F^{(2)}(K)$, with the first two lying in $f_1(K)$ and the last two in $f_2(K)$.    Continuing in this way, the number of components of $F^n(K)$ goes to infinity with $n$.  
Therefore $A_F$ has infinitely many components.
\end{proof}

\begin{example}[An Attractor with Infinitely many Connected Components that is not Totally Disconnected]  \label{ex:shooter}
Such an example is shown in Figure~\ref{fig:shooter}.  This is the attractor of an affine IFS consisting of nine functions indicated by the different gray shades.  
 On the other hand, the attractor $A$ of any contractive IFS consisting of exactly two functions is know to be either connected or totally disconnected \cite{MB}.
\end{example}

\begin{figure}[htb]  
\vskip -5mm
\includegraphics[width=5cm, keepaspectratio]{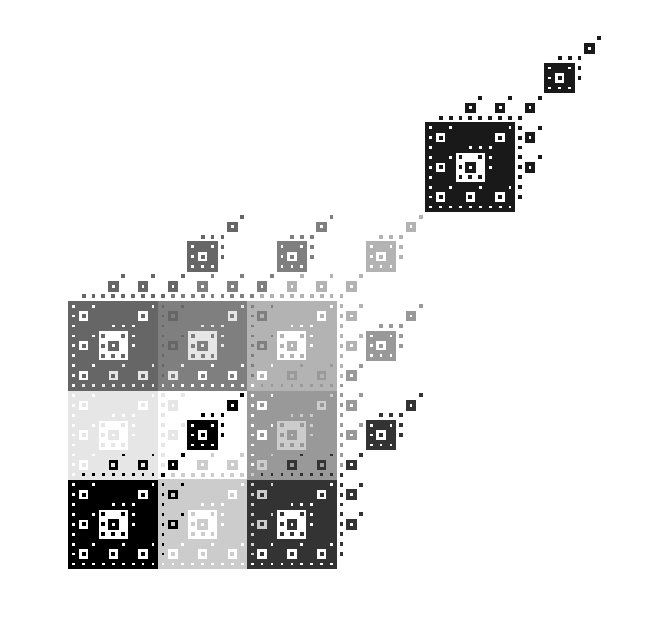}
\caption{An attractor with infinitely many connected components that is not totally disconnected.}
\label{fig:shooter}
\end{figure}

\begin{example}[A One-Parameter Affine Family  $F_t$ such that $A_t$ is Disconnected for all $t \in [0, t_0)$]  
\label{ex:disAt}
Consider the one-parameter affine family 
\[F_t = \{f_t  \begin{pmatrix} x \\ y \end{pmatrix} = t\, \begin{pmatrix} 1 & 0 \\ 0 & \frac{1}{10} \end{pmatrix} \begin{pmatrix} x \\ y \end{pmatrix}, \; g_t    \begin{pmatrix} x \\ y \end{pmatrix} = t\, \begin{pmatrix} \frac{1}{10} & 0 \\ 0 & \frac{1}{10} \end{pmatrix} 
\begin{pmatrix} x \\ y \end{pmatrix} + \big (1-\frac{t}{10} \big) \, \begin{pmatrix} 1 \\ 1 \end{pmatrix}. \] 
By Theorem~\ref{thm:jsr}, if $t \geq 1$, then $F_t$ has no attractor.  And if $t< 1$, then the attractor of
$F_t$ is disconnected.  This can been seen because, if $K$ is the unit square with vertices $(0,0), (1,0), (0,1), (1,1)$,
then for all $t\in [0, 1)$, we have $F_t(K) \subset K, \, (0,0) \in f_t(K), \, (1,1) \in g_t(K)$.  However, $f_t(K)\cap g_t(K) = \emptyset$.  
\end{example}

\begin{example}[A One-Parameter Similarity Family with no Single Threshold for Connectivity of the Attractor]
\label{ex:2}
An examination of the Mandelbrot set for the family given in \eqref{eq:man}, in particular the region of  the ``ram's horn" about a third of the way up from the horizontal  in Figure~\ref{fig:mandel},  shows that, 
even for the simple one-parameter family of the form \eqref{eq:man}, there can exist 
$0 <t < t'< t'' < 1$ such that $A_{ t}$ and  $A_{ t''}$ are connected, but $A_{ t'}$ is not connected.  
Thus there can, in general, exist no single threshold for connectivity for a one-parameter affine family, even a family consisting of 
two similarities. 
\end{example}

Because of  Example~\ref{ex:disAt},  it is assumed  in Theorem~\ref{thm:sim} below that the functions in the IFS are similarities. 
 Because of Example~\ref{ex:2}, we prove the existence of two thresholds $\tau$ and $\tau'$ such that, for $t$ greater then $\tau$, the attractor $A_t$ of a one-parameter
similarity family is connected, and for $t$ less than $\tau'$, the attractor $A_t$ of an affine, but not linear, family is disconnected.  A slightly weaker notion of connectivity, called {\it weak connectivity}, is then introduced, for which 
a single threshold is proved; see Theorem~\ref{thm:wc}.

\begin{theorem} \label{thm:sim} Given a one-parameter similarity family $F_t$, there exists a real number $\tau$ with $\tau \in [0,t_0)$ such that the attractor $A_t$ is connected for all $t \in  (\tau,t_0)$.  
\end{theorem}

\begin{proof} By remark~\ref{rem:noloss}, there is no loss of generality in  assuming that $t_0 = 1$.   For a similarity $f$, let $s_f$ denote the scaling ratio of $f$.  
We first claim that, for any partition $(F^*,H)$ of the set of functions of a similarity IFS $F$ in $\R^d$ 
with $F^*\neq \emptyset, \, H \neq \emptyset$, there exist functions $g\in F^*$ and $h\in H$ such that $(s_g)^n + (s_h)^n > (\rho(F))^n = 1$.  
To see this, let $g$ be the function in $F$ with the maximum scaling ratio, and assume without loss of 
generality that $g\in F^*$.  If $h$ is any function in $H$, then  $1 = (s_g)^n  < (s_g)^n + (s_h)^n$.

Let $\tau = 1/((s_g)^n + (s_h)^n)^{1/n} <1$ so that if $1 > t > \tau$, then  $(s_g)^n + (s_h)^n > 1/t^n > 1$,
which implies that $(s_{g_t})^n + (s_{h_t})^n  = ( t s_g)^n + (t s_h)^n > 1$.  For $t$ such that $1 > t > \tau$, we claim that 
 $F_t$ has a connected attractor.  Assume, by way of contradiction, that this is not the case. 
This imples that either (1) $F_t$ has no attractor or (2) $F_t$ has an attractor $A_t$, but $A_t$ is not connected.  
Case 1 is not possible by Corollary~\ref{cor:existence} since  
$\rho( F_t ) <1$.  Referring to Theorem~\ref{thm:jsr}, let $K$ be a compact convex set containing the attractor $A_t$ such that $F_t(K) \subset K$. 
Letting $K_j := F_t^{(j)}(K)$, this implies that $K_{j+1} \subset K_j$ for all $j$.   
Since $A_t = \lim_{j\rightarrow \infty}
 K_j$ and the intersection of a nested sequence of compact connected sets is connected, there must exist an
integer $j$ such that $K_j$ is connected, but $K_{j+1} = \bigcup_{f\in F}\, f_t(K_j)$ is not.  
Therefore there is a partition $(F^*,H)$ of $F$  with $F^*\neq \emptyset, \, H \neq \emptyset$, such that
$g_t(K_{j}) \cap h_t(K_{j}) = \emptyset$ for all $g\in F^*$ and $h \in H$.  However, 
the existence of functions $g,h$ such that $(s_{t g})^n + (s_{t h})^n > 1$ (by the paragraph above) implies that 
 $volume \, g_t(K_j)+ volume \, h_t(K_j) > volume \, K_{j+1}$, which is impossible since $g_t(K_j) \cap h_t(K_j) = \emptyset$
and $K_{j+1} \subset K_{j}$.
\end{proof}

\begin{example}[An Affine Family $F_t$ for which the Attractor $A_t$ is Connected for all $t\in [0,t_0)$] 
\label{ex:surprise2}
Any quasi-linear one-parameter family is such an example (see Example~\ref{ex:surprise} and the general construction following it).  
For a quasi-linear  family the attractor  $A_t$ is, by Proposition~\ref{prop:quasi}, a single point for all $t\in [0,t_0)$, hence connected.
\end{example}

\begin{theorem}  \label{thm:aff2} If $F_t = \{t f_i(x) + q_i\, : \, 1\leq i \leq N\}$  is a one-parameter affine, but not quasi-linear, family, then there is a real number $\tau >0$ such that $A_t$ is disconnected for all $t\in [0,\tau)$.
\end{theorem}

\begin{proof} 
First assume that not all the $q_i$ are equal.  If $B$ is a ball containing the origin and all the $q_i$  in its interior, and if $t$ is sufficiently small, say $t < \alpha_1$, then $(f_i(B)+q_i) \cap (f_j(B)+q_j) = \emptyset$ for all $i,j$ such that $q_i\neq q_j$ and $f_i(B)+q_i \subset B$ for all $i$.  Therefore, for $t <\alpha_1$, we  have $F_t(B) \subset B$ and $F_t(B)$ is disconnected.  This is sufficient to insure that $A_t$ is disconnected.

From the paragraph above, we may assume that all the $q_i$ are equal, that the one-parameter family has the form  $F_t= \{  t\, f_i(x) + q \, : \, 1\leq i \leq N\}$.  Let $B$ be a ball containing $q$ in its interior.  If $t$ is sufficiently small, say $t < \alpha_2 \leq
\alpha_1$, then $t \, f_i(B) \subset B$ for all $i$; hence, for
$t < \alpha_2$, we have $F_t(B) \subset B$.  Since $A_t = \bigcap_{n\geq 0} F_t^{(n)}(B)$, to prove that $A_t$ is disconnected,
it is sufficient to show that  $F_t^{(n)}(B)$ is disconnected for some integer $n$.  Assume, by way of contradiction, that
 $X_n:= F_t^{(n)}(B)$ is connected for all $n$.  For any compact set $X\subset \R^d$, denote its diameter by $D(X)$,  i.e., 
the largest distance between any two points of $X$.  For each $i$, there is a constant $c_i$ such that, for any $X\subset \R^d$,
it is the case that $D(f_i(X)) \leq c_i \,D(X)$.  Therefore there is a constant $c$ (depending on $N$) such that, if $\bigcup_{i=1}^N f_i(X)$ is connected, then  $D(\bigcup_{i=1}^N f_i(X)) \leq c \,D(X)$, and therefore  $D(\bigcup_{i=1}^N t\, f_i(X)) \leq c t \,D(X)$.  Under our assumption that $X_n$ is connected for all $n$, we now have $D(X_{n+1}) \leq c t D(X_n)$, which implies that $D(X_n) \leq (c t)^n D(B)$.  For $t< \alpha_3 := \min\{c, \alpha_2\}$, this implies that the diameter of $A_t = \lim_{n\rightarrow \infty} X_n$ is $0$; hence $A_t$ is a single point, say $x_t$.   Thus $x_t$ is the fixed point of the function $t f_i(x)+q$ for all $i$, which contradicts that $F_t$  is not quasi-linearity.
\end{proof}

\begin{definition}  \label{def:wc}  A hyperplane $H$ in $\R^d$ {\it separates} a compact set $S$ if $S \cap H = \emptyset$ but $S$ has non-empty intersection with  both half-spaces determined by $H$.
A compact subset $S$ of $\R^d$ is {\it strongly disconnected} if there exists a hyperplane that separates $S$.  
A compact set $S$ is
{\it weakly connected} if there is no such hyperplane.  For a compact set $S$, call a maximal weakly connected subset of $S$ a {\it weak component} of $S$, and call the convex hull of a weak component a {\it convex component} of $S$.  
Figure~\ref{fig:weak} shows the weak components of a set.  
\end{definition} 

\begin{figure}[htb] 
\vskip -3mm 
\includegraphics[width=13cm, keepaspectratio]{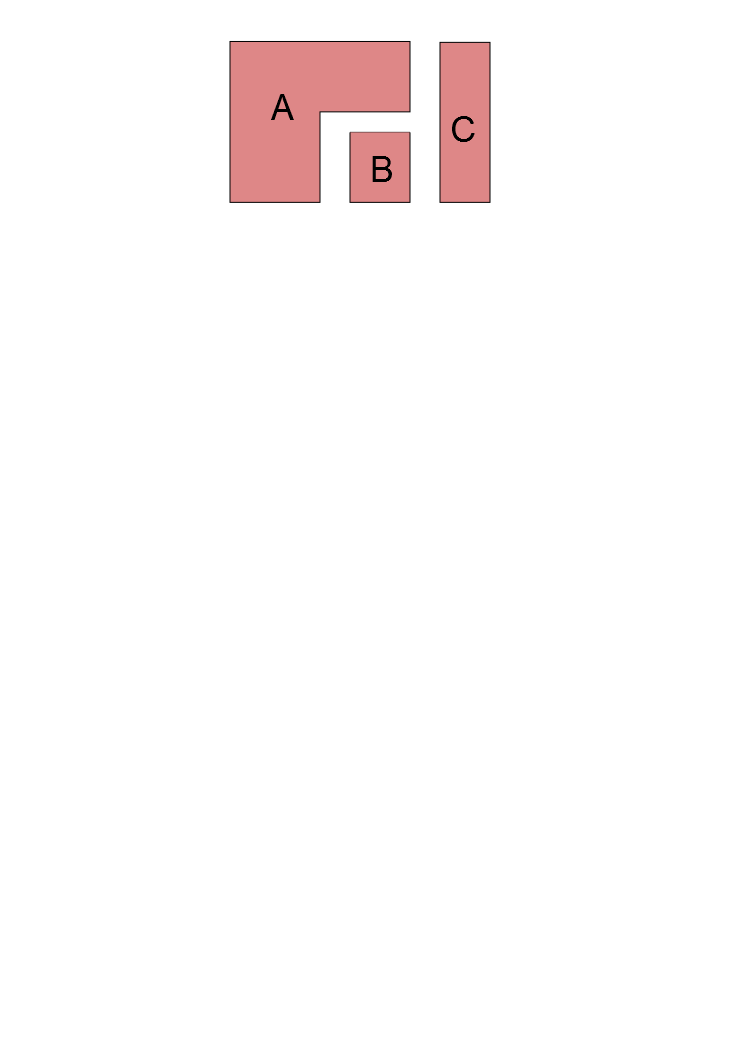}
\vskip -14.3cm
\caption{The sets $A\cup B $ and $C$ are the weak components.}
\label{fig:weak}
\end{figure}

\begin{proposition}  \label{prop:components}
The set of weak components of a compact set $S$ form a partition of $S$.
\end{proposition}

\begin{proof}  Assume, by way of contradiction, that $C_1$ and $C_2$ are weak components of $S$
with non-empty intersection.  By the definition of a weak component,  $C_1\cup C_2$ is not weakly connected
and hence there is a hyperplane $H$ that separates $C_1 \cup C_2$.  Since $C_1$ and $C_2$ are weakly connected, $C_1$ lies in one of the two halfspaces determines by $H$ as does $C_2$ .  If $C_1$ and $C_2$ are contained in the same
half-space, then $H$ does not separate $C_1 \cup C_2$, a contradiction.  If  $C_1$ and $C_2$ are contained 
different half-spaces, then $C_1$ and $C_2$ have empty intersection, again a contradiction.
\end{proof}

The proofs of the following lemmas  are straightforward.  For a compact set $C$ denote its convex hull by $conv\, C$.

\begin{lemma} \label{lem:sep}  Given an affine function $f$,  a compact set $C$ can be separated by a hyperplane if and only if $f(C)$ can  be separated  by a hyperplane. 
\end{lemma}

\begin{lemma} \label{lem:convex}   If $K$ is a compact set, $\{C_i\}$ a finite collection of compact sets,  and $f$ is an affine function, then 
\begin{enumerate}
\item $f(\text{conv } K) \subseteq \text{conv } f(K)$
\item $\bigcup_i \text{conv } C_i \subseteq \text{conv } (\bigcup_i C_i)$.
\end{enumerate}
\end{lemma}

\begin{lemma} \label{lem:SD}  If $K$ is compact set  and $\widehat K$ denotes the union of the convex components
of $K$, then 
\begin{enumerate}
\item $\widehat {\widehat K} = \widehat K$,
\item $\widehat K$ is strongly disconnected if and only if $K$ is strongly disconnected.
\end{enumerate}
\end{lemma}

\begin{theorem}  \label{thm:wc}  Let $F_t$ be a semi-linear, but not linear, affine one-parameter family. 
There is a number $\tau \in (0,t_0]$  such that the attractor $A_t$ is weakly connected for $t\in (\tau,t_0)$
and  strongly disconnected for $t\in (0,\tau)$.   
Moreover, $\tau < t_0$ if the functions in $F_t$ are similarities.
\end{theorem}

\begin{proof}  Let $F_t  = \{  t\, f_i(x) + q_i \, : 1\leq i \leq N \}$.  
Theorem~\ref{thm:jsr} guarantees that, for $t \in [0,t_0)$, there is a convex body $K$ (depending on $t$) such that $F_t(K) \subset K^o$.   
That $F_t$ is semi-linear implies, by Proposition~\ref{prop:semi}, that the fixed point of  $t\, f_i(x) + q_i$ is $q_i$ for all $t \in [0,t_0)$. 
That $F_t$ is not linear implies that not all  the $q_i$ are equal.  
Let $H$ be a hyperplane that separates the set $\{q_i \, : i=1,2,\dots , N\}$  of fixed points. 
If $t$ is sufficiently small, then the set $\{t f_i(K) +q_i \, : \,  i=1,2,\dots , N\}$ is also separated by $H$.  
Since $K \supset F(K) \supset F^{2}(K) \supset \cdots$, the attractor $A_t$ of $F_t$ is also separated by $H$.  Therefore  $A_t$ is strongly disconnected for $t$ sufficiently small.  Let $\tau$ be the supremum of those $t$ such that $A_t$ is strongly disconnected. 

Since $A_t$ is weakly connected for $t\in (\tau,t_0)$, it only remains to show that $A_t$ is 
strongly disconnected for $t \in (0, \tau)$.  For this it is sufficient to show the following: if $A_{t'}$ is strongly disconnected and $t < t'$,
then  $A_{t}$ is also strongly disconnected.
Let $K$ be the convex body such that $F_{t'}(K) \subset K^o$; let $K_n = F_{t'}^{n}(K)$; and let 
$C_n = F_{t}^{n}(K)$.  Thus,  with respect to the Hausdorff metric, $K_n \rightarrow A_{t'}$ and
 $C_n \rightarrow A_{t}$. Note that the fixed points $q_i, \, i=1,2, \dots, N,$ all lie in $K$.  
 Because $K$ is convex, if $K$ contains the origin and $L$ is any linear map, then 
$t L(K) \subset t' L(K)$ if $t < t'$.  It then follows from the definition of semi-regular that $t f_i(K) + q_i \subset
t ' f_i(K) + q_i $.
Since $C_0 = K$, we have $C_1 = F_{t}(C_0) = F_{t} (K) = \bigcup_{i=1}^N  t f_i(K) + q_i  \subset 
\bigcup_{i=1}^N t' f_i(K) + q_i  = F_{t'} (K) = K_1 \subset K = C_0$, 
which implies inductively that 
$C_0 \supset  {C_1} \supset  {C_2}\supset  {C_3} \supset\cdots $.   Since  $A_{t'}$ is assumed strongly disconnected, there is a least natural number  $n$ such that $K_n$ is strongly disconnected.  We claim that $C_n$
is also strongly disconnected, which would  imply that  $A_{t}$ is strongly disconnected, thus
proving the theorem.   

 Denote by $\widehat {K_n}$ and $\widehat {C_n}$ the union of the convex components
of $K_n$ and $C_n$ respectively. To simplify notation, let $f_{i,t}(x) := t f_i(x)+q_i$.  
To prove the claim, first note that $\widehat {C_{j}}$  and $\widehat {K_{j}}$ are connected for $j=1,2,\dots, n-1$.  For these values of $j$, by Lemma~\ref{lem:sep}, $f_{i,t'}(K_{j})$ and $f_{i,t}(C_{j})$ are weakly connected 
for all $i$.  We have previously in this proof shown that $C_1 \subset K_1$, hence $\widehat C_1 \subset
\widehat K_1$.  Proceeding inductively, 
if $ \widehat {C_j} \subseteq 
 \widehat {K_j}$,  then
\[C_{j+1} \subseteq \bigcup_{i=1}^N f_{i,t}(\widehat{C_j} ) \subseteq \bigcup_{i=1}^N f_{i,t}(\widehat{K_j} )
\subseteq\bigcup_{i=1}^N f_{i,t'}(\widehat{K_j} ) \subseteq \widehat {K_{j+1}}. \]
The last inclusion above is true by Lemma~\ref{lem:convex} because $f_{i,t'}(K_j) 
\subseteq K_{j+1}$ implies that $f_{i,t'}(\widehat {K_j}) = \widehat {f_{i,t'}(K_j) } 
\subset \widehat{K_{j+1}}$ for all $i$.  The second to 
last inclusion is true because $K_j$ contains the fixed points of all $i$ and, using statement (3) of Lemma~\ref{prop:props},
as in the paragraph above, $f_{i,t}(K_j) \subseteq f_{i,t'}(K_j)$ for all $i$.  
From $C_{j+1} \subseteq \widehat {K_{j+1}}$ it follows from Lemma~\ref{lem:SD} that
$\widehat {C_{j+1}} \subseteq \widehat {\widehat {K_{j+1}}}  =  \widehat {K_{j+1}}$.
Applying this  inclusion to the case $j = n-1$  gives $\widehat {C_{n}} \subseteq  \widehat {K_{n}}$.
Since $K_n$ is strongly disconnected, so is $\widehat {K_{n}}$ by Lemma~\ref{lem:SD}.  But because
$\widehat {C_{n-1}}$  and $\widehat {K_{n-1}}$ are both connected and $\widehat {C_{n-1}} \subseteq\widehat {K_{n-1}}$, each connected component of $\widehat {K_{n}}$ contains a connected component of  $\widehat {C_{n}}$. 
Therefore  $\widehat {K_{n}}$ strongly disconnected implies that  $\widehat {C_{n}}$ is also strongly disconnected,
which is the result required.

That $\tau < t_0$ for a similarity family follows from Theorem~\ref{thm:sim}, which states that there is a $\tau' < t_0$ such that
$A_t$ is connected, and hence weakly connected, for $t\in (\tau ', t_0)$.  Therefore $\tau < t_0$.   
\end{proof}

\section{Threshold for the Appearance of an Attractor with Non-Empty Interior} \label{sec:interior}

\begin{lemma} \label{lem:two}
Let $F$ be an affine IFS and $G \subset F$.  If $G$ has an attractor, then so does $G$, and $A_F \subseteq A_G$.  In particular,
if  $A_F$ has non-empty interior, then $A_G$ also has non-empty interior.  
\end{lemma}

\begin{proof}   If $F$ has an attractor, then  by part (3) of Theorem~\ref{thm:jsr} $G$ also has an attractor.     Moreover, since $G^n(A_F) \subseteq A_F$ for all $n$, we
have $A_G = \lim_{n\rightarrow \infty} G^n(A_F) \subseteq A_F$.    
\end{proof}

\begin{theorem}  \label{thm:empty}   If $F_t = \{t f_i(x) + q_i\, : \, 1\leq i \leq N\}$  is a one-parameter affine family, 
then there is a real number $\tau >0$ such that $A_t$ has empty interior for $t\in (0,\tau)$.
\end{theorem}

\begin{proof}  By Lemma~\ref{lem:two} it suffices to prove the theorem when $F_t$ consists of two functions $f_t(x) = t f(x) +p$ and $g_t = t g(x) + q$.  
Let $\mu(B)$ denote the Lebesgue  measure of a compact set $B$.  Then $\mu( f_t(B) ) = t^d \det(L_f) \mu(B)$, where $L_f$ is the linear part of $f$; similarly
for $g_t$.  Since $A_t = f_t(A_t)\cup g_t(A_t)$, we have $\mu(A_t) \leq  \mu(f_t(A_t)) +\mu( g_t(A_t)) =  t^d (\det(L_f) + \det(L_g)) \mu(A_t)$.  If $A_t$ has
non-empty interior, then $\mu(A_t) > 0$, which implies that $ t^d (\det(L_f) + \det(L_g)) > 1$.  This is not possible if $t$ is sufficiently small.  
\end{proof}

It is not hard to construct affine families $F_t$ for which there is a proper affine subspace $W \subsetneq \R^d$ that is invariant under $F_t$ for all $t\in [0,t_0)$.  For such a family the
attractor satisfies $A_t \subset W$, which implies, trivially, that the interior of $A_t$ is empty.  Call a family $F_t$ 
with such an invariant proper affine subspace {\bf degenerate}, otherwise
{\bf non-degenerate}.  

\begin{example}[An Non-Degenerate One-Parameter Affine Family $F_t$ such that $A_t$ has Empty Interior for all $t\in [0,t_0)$]  
\label{ex:allempty}
 Consider the one-parameter similarity IFS family 
\begin{equation} \label{eq:empty} \begin{aligned}
 F_t &= \Bigg \{ t f_1\begin{pmatrix} x\\ y \end{pmatrix}, \; t f_2\begin{pmatrix} x\\ y \end{pmatrix}  + \begin{pmatrix} 1 \\ 0 \end{pmatrix}\Bigg \}, \qquad \qquad \text{where} 
 \\ \\ 
f_1\begin{pmatrix} x\\ y \end{pmatrix} =  &\begin{pmatrix} 0 & -1 \\ 1 & 0 
\end{pmatrix}\,\begin{pmatrix} x\\ y \end{pmatrix}, \qquad
 f_2\,\begin{pmatrix} x\\ y \end{pmatrix} =\alpha \,  \begin{pmatrix} x\\ y \end{pmatrix} - \alpha\,\begin{pmatrix} 1 \\ 0 \end{pmatrix}
\end{aligned} \end{equation}
and $0<\alpha <1/3$. Note that the threshold $t_0$ for the existence of an attractor is $1$ in this example.  

\begin{theorem}  \label{thm:empty2} The attractor $A_t$ of the family $F_t$ of Equation~\eqref{eq:empty} has empty interior for all $t\in [0,t_0)$. 
\end{theorem}

\begin{proof}  It is easily verified that $F_t(B_t) \subset B_t$, where $B_t$ is the compact set consisting of four solid quadrilaterals depicted in Figure~\ref{fig:empty}.  
Denote by $g_1$ and $g_2$ the two functions in $F_t$.  Note that $g_1(B_t) = B_t$ and $g_2(B_t)$, shown in red, is contained in the disk of radius $t \alpha$ centered at 
$(1-t \alpha),0)$.  

By way of contradiction, assume that $A_t$ has non-empty interior.  Then there is a ball contained in the interior, and hence there is an arc of a circle centered
at the origin contained in $A_t$.  Let $\gamma$ be the arc of greatest length contained in $A_t$.  Denote by $g_1$ and $g_2$ the two functions in $F_t$.  Since $g_1$ and $g_2$
are similarities, there must exist two circular arcs $\gamma_1$ and $\gamma_2$ in $A_t$ such that  $\gamma \subseteq g_1(\gamma_1) \cup g_2(\gamma_2)$.  Since 
 $g_1$ and $g_2$ are contractions, the radii $r_1, r_2$ of the circles corresponding to  $\gamma_1$ and $\gamma_2$  must be greater than the radius $r$ of the circle corresponding to  $\gamma$.  Either  $\gamma \subseteq g_1(\gamma_1)$ or $\gamma_2 \neq \emptyset$.   If $\gamma_2 \neq \emptyset$, then (1) $\gamma$ is contained in the
rightmost quadrilateral of $B_t$ and (2) by the paragraph above $r < r_2 \leq t \alpha$, and (3) for $\gamma$ to be in the range of $g_2$ it must be the case that 
$r> 1-t\alpha - t^2 \alpha$.
But $t\alpha >  1-t\alpha - t^2 \alpha$ implies that $\alpha > 1/3$, a contradiction.  Therefore $\gamma \subseteq g_1(\gamma_1)$.  In this case $\gamma$ is the
image of an arc in $A_t$ of length greater than that of $\gamma$, a contradiction.   
\end{proof}

\begin{figure}[htb]  
\includegraphics[width=8cm, keepaspectratio]{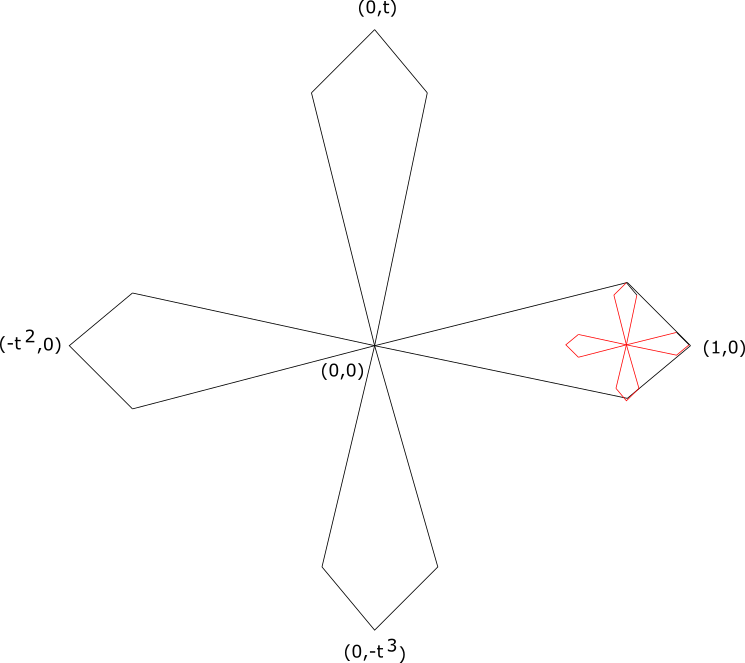}
\caption{The set $B_t$ in the proof of Theorem~\ref{thm:empty2}.}
\label{fig:empty}
\end{figure}

\end{example}

\begin{lemma}  \label{closure} If the attractor $A$ of an affine IFS has non-empty interior, then  $A$ is the closure of its interior.
\end{lemma}

\begin{proof}  By definition, each function in the IFS $F$ has non-singular linear part.  Therefore each function in $F$ takes the interior of $A$ into the interior of $A$.  
If a point $x$ lies in the closure of the interior of $A$, call it $\overline A$, then $f(x)$ must lie in  $\overline A$.  Therefore the Hutchinson operator $F$ takes $\overline A$ into $\overline A$, 
which implies that $A = \lim_{n\rightarrow \infty} F^n(\overline A) \subseteq \overline A \subseteq A$.  
\end{proof}

\begin{definition} \label{def:tame}
By a {\it cone} in $\R^d$ in this paper, we mean a right circular non-degenerate (not a line segment) cone. 
 For a compact set $K$ with non-empty interior, define a boundary point  $x$ to be a {\it cusp} of $K$ if there is no cone with apex at $x$ contained in $K$.
For an affine IFS $F$ that has an attractor $A$, Theorem~\ref{thm:jsr} implies that each functions in $F$ has a unique fixed point
that is contained in $A$.   Let $Z$ denote the set of fixed points of the functions in $F$.  
   Call $F$  {\bf tame} if not all points in $Z$ are cusps of $A$.  In particular, if there is even one point of $Z$ that lies in the interior of $A$, then $F$ is tame.  (We have no example of an IFS that has an attractor with non-empty interior  and is not
tame.)  Call a one-parameter affine family $F_t$ {\bf tame}  if  $F_t$ is tame for all $t$ such that the attractor $A_t$ has
non-empty interior.  
\end{definition}

\begin{theorem}  Let $F_t$ be a tame, semi-linear one-parameter similarity family.
 Then there is a real number $t_2 >0$ such that  $A_t$  has empty interior for $t\in [0,t_2)$ and non-empty interior for 
$t\in (t_2,t_0 )$.  
\end{theorem}

\begin{proof}  If $A_t$ has empty interior for all $t\in [0, t_0)$ (which is possible by Example~\ref{ex:allempty} and Theorem~\ref{thm:empty2}), then $t_2 = t_0$
satisfies the statement of the theorem.

So we may assume that there is a $t$ such that $A_t$ has non-empty interior.   Let $t_2$ be the greatest lower bound of those $t$ such that the interior of $A_t$ is non-empty.    It is now sufficient to show that, if $0<t'<t''<t_0$, and $A_{t'}$ has non-empty interior, then $A_{t''}$ also has non-empty interior.  Let $x$ be the fixed point of a
function $f_t(x) = t f(x)+q \in F_t$ such that $x$ is not a cusp of $A_t$.  Then there is a cone $Y$ centered at $x$ that is contained in $A_t$.  The cone $Y$ can
be chosen small enough so that, by the definition of an attractor, $Y \subset f_{t'}(Y)$.  Because the family is semi-linear and $f_t$ is a similarity with fixed point $x$ for
all $t<t_0$, we have $Y \subset f_{t'}(Y) \subset f_{t''}(Y) \subset F_{t''}(Y)$.  This implies that $ F_{t''}(Y)  \subset F_{t''}^2(Y)$.  Proceeding inductively,
$Y \subset F_{t''}(Y)\subset F_{t''}^2(Y) \subset F_{t''}^3(Y) \cdots$.  Because $A_{t''} = \lim_{n\rightarrow \infty} F_{t''}^n(Y)$, we have $Y\subset A_{t''}$.
Thus  $A_{t''}$ has non-empty interior since it contains the (non-empty) interior of cone $Y$.
\end{proof}

\begin{lemma} \label{lem:int1} Let $A$ be the attractor of an IFS $F$.   If there is a ball $B \subseteq \R^d$ and an integer $n_0$ such that $B \subseteq F^{n_0}(B)$, 
then $A$ has non-empty interior. 
\end{lemma}

\begin{proof}  If $n_0$ is such that $B \subseteq F^{n_0}(B)$, then inductively $B \subseteq F^{m n_0}(B)$ for all positive integers $m$.  Since 
$A = \lim_{n\rightarrow \infty}  F^n(B)$, it must be the case that $B \subseteq A$.
\end{proof}

Let $F_t  = \{  t\, f_i(x) + q_i \, : 1\leq i \leq N \}$ be a semi-linear one-parameter similarity family, and let  $F = \{f_1, f_2, \dots, f_N\}$.  
Let $\rho$ be the maximum
scaling ratio of the similarities in $F$, and denote by $\widehat F = \{\widehat {f_1},\widehat {f_2}, \dots, ,\widehat {f}_{\widehat N}\}$ the subset of $F$ consisting of those similarities in 
$F$ with scaling ratio $\rho$.   Let
\[{\widehat F_t }=  \{  t\, \widehat {f_i}(x) + \widehat{q_i}\, :\, i=1,2,\dots, \widehat N\}\]  
denote the corresponding subset of $F_t$.
By Proposition~\ref{prop:semi},  each function $t_0 \widehat f_i(x) + \widehat{q_i} \in \widehat{F_{t_0}}$ is an isometry with fixed point $\widehat{q_i}$.

\begin{theorem} \label{thm:nonempty}
 Let $F_t  = \{  t\, f_i(x) + q_i \, : 1\leq i \leq N \}$ be a semi-linear, but not linear, one-parameter similarity family.  With notation as above, let  $H_t$ be any subset of
 $\widehat {F_t}$ having the property that the $\widehat{q_i}$ are all equal, i.e., $H_t$ has the form $H_t  = \{  t\, h_i(x) + q_0\, : 1\leq i \leq k \}$ for some $q_0$. 
 If  the orbit of a point on the unit sphere centered at $q_0$ under the action of the set $H_{t_0}$ of isometries  is dense on this unit sphere,  
then there is a $\tau < t_0$ such that the attractor $A_t$ of $F_t$ has non-empty interior for all $t\in (\tau, t_0)$.
\end{theorem}

\begin{proof} Referring to Remark~\ref{rem:noloss}, we assume without loss of generality that $t_0 = 1$, which implies that the maximum scaling ratio of
among the linear parts of the functions in $F_t$ is $1$. We also assume, without loss of
generality,  that $q_0$ is the origin.  Since $F_t$ is not linear, not all functions in $F_t$ have the same fixed point.  Let $g_t \in F_t$ be such that the fixed point of $g_t$ is not 
the origin $0$.  Denote the similarity ratio of $g$ by $0 < s <1$.  Choose $\epsilon > 0$ and $\pi/ 2 > \theta >0$ such that 
\begin{equation} \label{eq:cone} 2(1-\epsilon) \cos \theta > 1 + (1-\epsilon)^2 (1-s^2).\end{equation}
  Note that this is possible if $\epsilon$ and $\theta$ are chosen sufficiently small.  Let $\mathcal C$ be
a set of infinite circular cones of central angle $\theta$ centered at the origin, such that the union of the cones in $\mathcal C$ is $\R^d$.  

Let $p = g_{1}(0)$; let $r = |p|$;
and let $B_r$ be the ball of radius $r$ centered at $0$.  Since the orbit under $H_{1}$ of a point on the unit sphere centered at $0$ is dense on this unit sphere, there is an integer
 $M$ such that every cone in $\mathcal C$ contains a point of $\bigcup_{k=0}^M F_{1}^{k} (p)$.  Let $\tau = (1-\epsilon)^{M+1}$. Let $1 >t > \tau$, i.e. $t <1$ such that $t^{M+1} > 1-\epsilon$.  By Lemma~\ref{lem:int1}, it is sufficient to show that  $B_r \subseteq F_t^{M+1}(B_r)$.

For any $h_t \in H_t$, note that $h_t^{M+1}(B_r)$ is a ball of radius $r t^{M+1} > r(1-\epsilon)$.  Therefore 
$B_{r(1-\epsilon)} \subset h_t^{M+1}(B_r) \subseteq F_t^{M+1}(B_r)$.  It now only remains to show that the shell $S$, the region between the
concentric spheres of radius $r$ and radius  $r(1-\epsilon)$, is covered by  $F_t^{(M+1)}(B_r)$.

The set $F_t^{k} g F_t^{M-k}(B_r) = F_t^{k} g (B_{r(1-\epsilon)}) \subset F_t^{M+1}(B_r)$ is the union of balls of radius $rst^M > rst^{M+1} > rs(1-\epsilon)$.  
Denote this set of balls
by $\mathcal B$.  The centers of these balls  lie in the shell $S$ and, for each point $x \in F_{1}^{k} (p)$, there is a ball $B_x \in \mathcal B$ whose center lies on the ray from $0$ through $x$.  Therefore, for any cone $C \in \mathcal C$, there is a ball $B \in \mathcal B$ whose center lies in $C$.   It now only remains to show that $S\cap C \subset B$.  
This is insured if, in Figure~\ref{fig:cone}, the radius of $B$ is greater than the length $d$.   Using the cosine law, we require
\[r^2 s^2 (1-\epsilon)^2 > d^2 = r^2 + r^2(1-\epsilon)^2 - 2 r^2 (1-\epsilon) \cos \theta,\]
which follows from Equation~\eqref{eq:cone}.
\end{proof}

\begin{figure}[htb]  
\includegraphics[width=4.5cm, keepaspectratio]{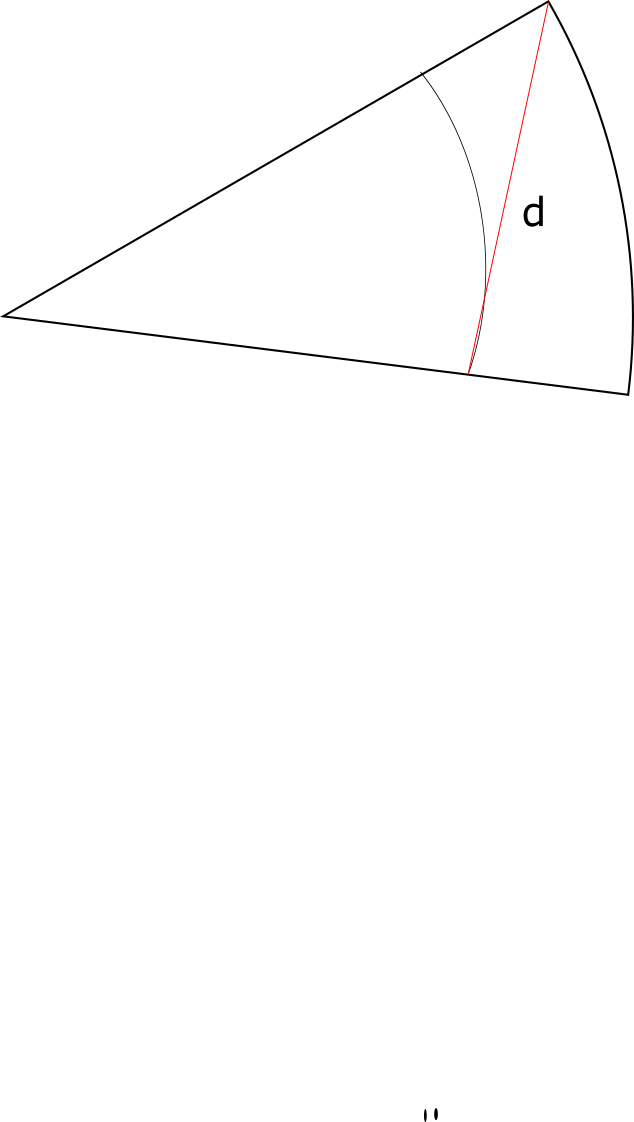} 
\vskip -5cm
\caption{The cone in the proof of Theorem~\ref{thm:nonempty}. }
\label{fig:cone}
\end{figure}

\begin{cor} \label{cor:nonempty} Let $F_t$ be a semi-linear, but not linear, similarity family in $\R^2$.  If there is a function in $F_t$ whose linear part has 
maximum scaling ratio among the functions in $F_t$  and whose rotation angle  is an irrational multiple of $\pi$, then there is a $\tau < t_0$ such
that the attractor $A_t$ of $F_t$ has non-empty interior for all $t\in (\tau, t_0)$.
\end{cor}

\section{Transition Attractors for Bounded Families} \label{sec:transition}

This section concerns the phenomena 
that can occur at the transition point $t_0$ between the 
existence and non-existence of an attractor. Two notions of a transition attractor for a one-parameter family $F_t$ are defined. 

\begin{definition} \label{def:bounded}
 Call a one-parameter similarity family $F_t$ {\bf bounded} if it is semi-linear and if there is a unique function, call it  
$\widehat {f}_t(x) = t \widehat {f}(x) + q_* \in F_t$,  
such that  $\widehat {f}$ has maximum scaling ratio. The motivation for the terminology ``bounded" is Theorem~\ref{thm:bounded} below.  By Proposition~\ref{prop:semi},  $q_*$ is a fixed point of $\widehat {f}_t(x)$ for all   
$t\in [0,t_0]$.   Call  $f^* = {\widehat f}_{t_0}$  the 
{\bf special function} associated with $F_t$,  and call $q_*$ the {\bf special fixed point}. Note that the special function $g$ is an isometry.  
\end{definition}



\begin{theorem} \label{thm:bounded}  If $F_t$ is a bounded one-parameter family, then there is a ball $B$ such that 
 then $F_t(B) \subset B$ for all $t\in [0, t_0]$.  
\end{theorem}

\begin{proof} 
Without loss of generality assume (Remark~\ref{rem:noloss}) that $t_0 = 1$.  Using Lemma~\ref{prop:props} and Proposition~\ref{prop:semi},
let  $F_t = \{ t \, L_i(x - q_i) + q_i : i=1,2,\dots, N\}$ be the bounded  one-parameter family.  The scaling ratio of $L_i$ is denoted $r_i$.   
Without loss of generality assume that the special fixed point is $0$, the origin.  
 Let $B$ be a ball centered at the origin of radius 
\[R > \max \Big \{  \frac{\max (|(I-L_i)( q_i)|,|q_i|)}{1- r_i} \,:\,  i=1,2,\dots, N, \, r_i < 1  \Big \}.\] 

For the  unique function  $\widehat {f}_t(x) = t \widehat {f}(x) +q_* \in F_t$  whose 
linear part has scaling ratio $1$ we have $|\widehat {f}_t(x) | = t |L_f(x)| \leq R$ for all $x \in B$ with equality if and only if $t=1$.

For any $f_t(x) = t f(x) + q \in F_t$  whose linear part $L$ has scaling ratio $r<1$, let $h(x) = f(x) + q$.  Then $f_t(x) = t\, h(x) + (1-t) q$.
For $t\in [0, t_0]$ and for all $x\in B$, we have
\[ \begin{aligned} |f_t(x)| &=  |t \, h(x) + (1-t) q|  \leq  t |h(x) |+ (1-t)\,|q| \leq  t |f(x)| + t |q|+ (1-t)\,|q|  \\
&\leq t |L(x)| +\big ( \, t \,|(I-L)(q)|+ (1-t)|q|\big ) \leq t r R + \max \{|(I-L)( q)|, |q| \} \\
&\leq t r R + (1-r) R \leq  t r R + (1- t\, r) R = R.\end{aligned} \]
\end{proof}

\begin{remark} Theorem~\ref{thm:bounded} may fail to hold  if there is more than one function whose linear part has maximum scaling ratio. This is the
case for the the one-parameter family of Example~\ref{ex:I}. Theorem~\ref{thm:bounded} may also fail to hold if the family is not semi-linear; any affine family for which $t_0 = 1$ that contains a function of the form $tx+a$ is an example.
\end{remark}

\begin{cor} For a bounded one-parameter family $F_t$, there is ball $B$ such that the attractor $A_t$ is contained in $B$ for all $t\in [0,t_0)$.  
\end{cor}

\begin{proof}  This follows immediately from the fact that $A_t = \bigcap_{n\geq 1} F^n_t(B)$, where $B$ is the ball in the statement of Theorem~\ref{thm:bounded}.
\end{proof}

For a bounded one parameter family  $F_t$, call a compact set $A^*$ {\bf an upper transition attractor} if $A^*$ is the limit of a sequence $A_{t_k}, \, k=1,2,\dots,$ of attractors with 
$t_k \rightarrow t_0$.  By the Balzano-Weierstrass theorem, there is at least one upper transition attractor.  For a one-parameter affine family $F_t$, Theorem~\ref{thm:continuity} states that the map $t\mapsto A_t$ is continuous on $(0,t_0)$, but Example~\ref{ex:I} shows
that it may not be uniformly continuous.   For the family $F_t$ in Example~\ref{ex:I}, however, the attractors $A_t$ grow in size without bound as $t\rightarrow t_0$.   Call a one-parameter affine family $F_t$  {\bf uniform} if the function $t\mapsto A_t$ is uniformly continuous on the interval $(0,t_0)$.  A basic result from analysis implies that a one-parameter affine family $F_t$ is uniform if and only if 
$\lim_{t\rightarrow t_0} A_t$ exists.  The following is the main open problem in this paper.  It is stated in three equivalent ways.  Corollary~\ref{cor:uniform} below verifies the conjecture in the $1$-dimensional case, and graphical evidence seems to support it in dimension $2$.  

\begin{conjecture}  \label{conj:lim} For a bounded one-parameter family $F_t$, the following three equivalent statements hold:
\begin{enumerate}
\item there is a unique upper transition attractor; 
\item $\lim_{t\rightarrow t_0} A_t$ exists;
\item $F_t$ is uniform.  
\end{enumerate}
\end{conjecture}

\begin{definition}  For a uniform one-parameter affine family $F_t$, the compact set
\[A^* := \lim_{t\rightarrow t_0} A_t\]
will be called {\bf the upper transition attractor} of $F_t$.  
\end{definition}  

For a set $X$, let $conv \,X$ denote the convex hull of $X$.

\begin{lemma} \label{lem:nested} Let $F_t, \, 0 \leq t < t_0,$ be a bounded one-parameter family, and let $A_t$ denote the attractor of 
$F_t$.   If $K_t = conv \, A_t$, then $K_s \subseteq K_t$ for all $0 \leq s \leq t < t_0$.  
\end{lemma}

\begin{proof}  Without loss of generality (see Remark~\ref{rem:noloss}), assume that $t_0=1$.
If $B$ is the ball in  Theorem~\ref{thm:bounded}, then $F_t(B) \subset B$ for all $0 \leq t <1$.    
Let $f_t(x) := t f(x) + q$ be any function in $F_t$.  
Since $f_s(x)$ and $f_t(x)$ are both contractive similarity transformations centered at 
$q\in B$ with the same linear part up to a constant, it must be the case that $f_s(B) \subseteq f_t(B)$ for all 
$f\in F$. 
  For $0 \leq t < 1$, let $K_t$ be the smallest convex set $K$ such that $F_t(K) \subseteq K$, i.e., 
the intersection of all convex sets with this property.  Because all the fixed points of functions in $F_t$ 
are contained in such a convex set, 
$F_t(K) \subseteq K$ implies that $F_s(K_t) \subseteq K_t$. Therefore
\[K_s = \bigcap \{ K \, : \,  F_s(K) \subseteq K, K \; \text{convex}\} \subseteq  K_t.\]
The facts that the functions are affine and that $F_t(A_t)=A_t$ implies that  $F_t(conv \, A_t) \subseteq  conv \, A_t$.
Since $K_t$ is the smallest convex set $K$ such that $F_t(K) \subseteq K$, we have $K_t \subseteq conv \, A_t$.  On the
other hand, since $A_t = \bigcap_{n\geq 0} F^n(K_t)$, we have $A_t \subset K_t$, which implies that
$conv \,A_t \subseteq K_t$ because $K_t$ is convex.   
Therefore $K_t = conv \, A_t$.
\end{proof}

\begin{cor} \label{cor:uniform}  A $1$-dimensional bounded one-parameter family $F_t$ is uniform; equivalently $A^* := \lim_{t\rightarrow t_0} A_t$ exists.  
\end{cor}

\begin{proof}  By Theorem~\ref{thm:sim} the attractor $A_t$ is connected, hence is a point or a line segment for $t$ sufficiently close to $t_0$.  
Therefore,  $A_t = K_t$ for those values of $t$. By Lemma~\ref{lem:nested}, $A_s \subseteq A_t$ for $s\leq t$, so that $A^* := \lim_{t\rightarrow t_0} A_t$, 
as a limit of bounded nested intervals (or a point if $F_t$ is linear), is itself an interval (or a point). 
\end{proof}

Let $F_t, \, 0 \leq t < t_0,$ be a bounded one-parameter family, and let $A_t$ denote the attractor of 
$F_t$.  Let $K_t = conv \, A_t$ and 
\[K^* = \overline{ \bigcup_{t < t_0} K_t} = \lim_{t\rightarrow t_0} K_t,\] where the bar denotes the closure.  The
union is nested by Lemma~\ref{lem:nested} and is bounded by Theorem~\ref{thm:bounded}.
  Therefore $K^*$  is compact and convex.  
Call $K^*$ the {\bf transition hull} of $F_t$.  
\vskip 3mm

For a bounded family $F_t  = \{  t\, f_i(x) + q_i \, : 1\leq i \leq N \}$, let 
 \[F^* := F_{t_0} =  \{  t_0 \, f_i(x) + q_i \, : 1\leq i \leq N \}.\] 
From Definition~\ref{def:bounded}, the special function $f^*$, with fixed point $q_*$ is the unique isometry in $F^*$.
A compact set $X$ will be called $(F^*,q_*)$-{\bf invariant} if $F^*(X) = X$ and $q_* \in X$.  

\begin{definition} \label{def:lower}  For a bounded one-parameter family 
\[A_* = \bigcap \{A \in \mathbb H(\R^d)\,:\, A \; \text{is $(F^*,q_*)$-invariant} \}\]  will be called the {\bf lower transition attractor} of $F_t$.  The terms ``upper" and
``lower" are used because, for any upper transition attractor $A^*$, it is a consequence of Theorem~\ref{thm:K} below that $A_* \subseteq A^*$.  
Subsequent examples in this section show that they are not necessarily equal.  
\end{definition}

\begin{theorem}    \label{thm:K}
If $F_t, \, 0 \leq t < t_0,$ is a bounded family and $A^*$ is an upper transition attractor, then the following hold:
\begin{enumerate} 
\item  $A^*$ is $(F^*,q_*)$-invariant;
\item $A_*$ is the unique minimal (with respect to inclusion)  $(F^*,q_*)$-invariant set;
\item $K^*$ is unique minimal  $(F^*,q_*)$-invariant convex set;
\item $conv \, A_* = conv \, A^* = K^*$;
\item $A_* = \overline{\bigcup_{n\geq 0} {F^*}^n(q_*)}$.
\end{enumerate}  
\end{theorem}

\begin{proof}    Let $A^* = \lim_{k\rightarrow \infty} A_{t_k}$ be an upper transition attractor, where $t_k \rightarrow t_0$.  

Concerning statement (1),
We know that $q_*\in A^*$ because all the fixed points of the functions in $F_t$ are contained in $A^*$. 
We must  show that $F^*(A^*) = A^*$.   From the fact that $F_t(A_t) = A_t$ for all $t \in [0,t_0)$, we have
\[\lim_{k\rightarrow \infty} F_{t_k} (A_{t} )= \lim_{k\rightarrow \infty} A_{t_k} = A^*.\]  Therefore it is sufficeint to show that 
$F^*(A^*) = \lim_{k\rightarrow \infty} F_{t)k} (A_{t_k} )$,
or $ f_{t_0}(A^*) =  \lim_{k\rightarrow \infty} f_{t_k}  (A_{t_k} )$ for all $f_{t_k} \in F_{t_k}$.
For all  $f_{t_k} \in F_{t_k}$, we have  $ f_{t_k} (A_{t_k} ) - f_{t_0}(A^*) =
\Big (  f_{t_k} (A_{t_k} ) -  f_{t_k} (A^* )\Big )+\Big ( f_{t_k} (A^*) - f_{t_0}(A^*) \Big )$.  The result then follows from 
$A^* =  \lim_{k\rightarrow \infty}   A_{t_k}$ and
the fact that $t\mapsto f_t$ is continuous on $[0,t_0]$.

Concerning statement (2), clearly $q_* \in A_*$.  To show that  $A_*$ is $F^*$-invariant, in one direction,  that $F^*(A) = A$ for all $A$ that are  $(F^*,q_*)$-invariant, implies that  
\[\begin{aligned}  F^*(A_*) &= F^*\Big (\bigcap \{A \,:\, A \; \text{is $(F^*,q_*)$-invariant} \}\Big ) \subseteq \bigcap \{ F^*(A) \,:\, A \; \text{is $(F^*,q_*)$-invariant}\} \\
&= \bigcap \{ A \,:\, A \; \text{is $(F^*,q_*)$-invariant}\} = A_*.\end{aligned}\]
In the other direction, 
 $F^*(F^*(A_*)) \subset F^*(A_*)$ implies that
$F^*(A_*) \in   \{A \in \R^d\,:\, q_* \in A \; \text{and} \;F^*(A) \subseteq A \}$, which in turn implies that $A_* \subseteq F^*(A_*)$.  Since $A_*$ is the
intersection of all $F^*$-invariant sets, it is the minimum.

Concerning  the equality  $conv \, A^* = K^*$ in statement (4), the fact that $A_t \subseteq K_t \subseteq K^*$ implies that $A^* \subseteq K^*$.  Now, from Lemma~\ref{lem:nested},
\begin{equation} \label{eq:c}  conv \, A^* = conv \big (  \lim_{k\rightarrow \infty} A_{t_k}\big ) = 
 \lim_{k\rightarrow \infty}     conv \, A_{t_k} = K^*.\end{equation} 

Concerning the equality $F^*(K^*)= K^* $  in statement (3), using using Equation~\eqref{eq:c},
\[ \begin{aligned} F^*(K^*) &= F^* \big ( conv \, A^* \big ) =   \bigcup_{g\in F^*}  g\big ( conv \, A^* \big ) \subseteq\bigcup_{g\in F^*} conv \, g(A^*) \\
& \subseteq conv \, F^*(A^*)= conv \, A^* = K^*.\end{aligned}\]
Then  $K^* \subseteq F^*(K^*)$ because one function in $F^*$ is an isometry. 

To show that $K^*$ is the minimum such set in statement (3), we claim that 
$K^{*} := \bigcap \{K \,:\,K \; \text{compact convex}, \, F^*(K) \subseteq K, \, q_* \in K\}$. 
Let $K^{**} := \bigcap \{K \,:\,K \; \text{compact convex}, \, F^*(K) \subseteq K, \, q_* \in K\}$.   Since $K^*$ is one set
in this intersection, clearly $K^{**} \subseteq K^{*}$.  In the other direction, let $K$ be any element of $K^{**}$, i.e., $K$ is a compact convex set such that 
$F^*(K) \subseteq K$ and $q_* \in K$.  Since $K$ is compact, $F_*(K) \subseteq K$ and $q_*\in K$, we have that $\lim_{n\rightarrow \infty} \widehat f^n_t(q_*) \in K$ for all $\widehat f \in F^*$  which have similarity ratio less than $1$.  But the limit converges to the fixed point of $\widehat f$.  Since the special fixed point $q_*$ 
of the special function $f^* \in F^*$ is contained in $K$ by assumption, all fixed points of functions in $F^*$, and hence all fixed points of functions in $F_t$ for all $t\in [0,t_0)$, are contained in $K$.  Because $K$ is a convex set containing all fixed points of the similarity functions in $F_t, \,  t\in [0,t_0),$ 
 the fact that $F^*(K) \subseteq K$ implies that $F_t(K) \subseteq K$ for all $t\in [0, t_0)$.  Since this is true for all $K\in K^{**}$, we have  $F_t(K) \subseteq K^{**}$.
Now $A_t = \lim_{n\rightarrow \infty} F^n _t(K) \subseteq K \subseteq K^{**}$, and hence $K^* = conv A^* \subseteq K^{**}$.

Concerning the equality $conv \, A_* = K^*$ in statement (4), $conv \, A_*$ is $(F^*,q_*)$-invariant because $A_*$ is $(F^*,q_*)$-invariant by statement (2) 
and $F$ is affine.  Therefore, by statement (3),
we have $K^* =  \bigcap \{K \,:\,K \; \text{convex}, \, F^*(K) \subseteq K, \, q_* \in K\} \subseteq conv \, A_*$.  In the other direction, we have
$A_* \subseteq A^*$  because $A_*$ is the smallest $(F^*,q_*)$-invariant set and $A^*$ is $(F^*,q_*)$-invariant by statement (1).  Therefore
$conv\, A_* \subseteq conv \, A^* = K^*$ by the first equality in statement (4).

Concerning statement (5), since $q_* \in A_*$ and $A_*$ is $(F^*,q_*)$-invariant, it must be the case that  $\overline{\bigcup_{n\geq 0} {F^*}^n(q_*)} \subseteq A_*$.  But 
since   $\overline{\bigcup_{n\geq 0} {F^*}^n(q_*)}$ is $(F^*,q_*)$-invariant and $A_*$ is the minimum $(F^*,q_*)$-invariant compact set,  $A_* \subseteq \overline{\bigcup_{n\geq 0} {F^*}^n(q_*)}$.
\end{proof}

 \begin{example}[A Bounded Family for which $A^* \neq A_*$]   For the bounded one-parameter family 
$F_t = \{(t/4)\,x, t (x-1)+1 \}$ on the real line $\R$, the threshold $t_0=1$.  For $t<1$ sufficiently close to $1$, the attractor is the unit interval.
Hence $A^* = \lim_{t\rightarrow t_0} A_t = [0,1]$.  However, by statement (5) of Theorem~\ref{thm:K},
$A_* = \{0\} \cup \{(1/4)^n\}_{n=0}^{\infty}$, a countable set.  
\end{example}

\begin{figure}[htb]  
\includegraphics[width=7cm, keepaspectratio]{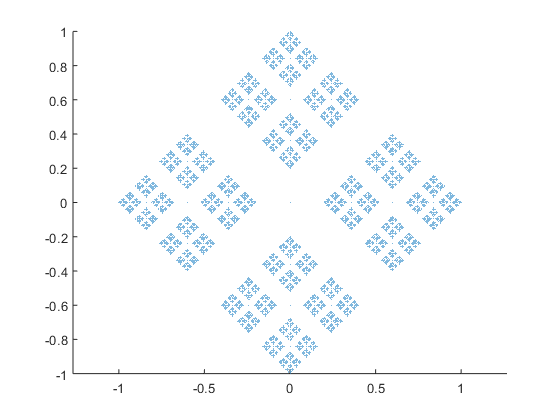} \hskip .5cm \includegraphics[width=7cm, keepaspectratio]{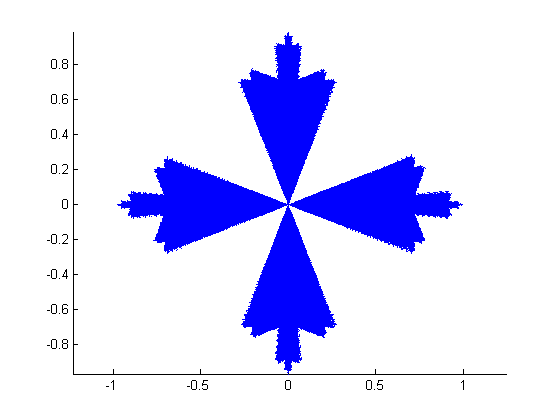} 
\caption{The lower and upper transition attractors of Example~\ref{ex:ta1}.}
\label{fig:tr1}
\end{figure}

\begin{example} \label{ex:ta1}  Consider  the bounded family
$F_t =\{t f_1(x), t f_2(x)+(1,0) \}$, where
\[f_1\begin{pmatrix} x\\ y \end{pmatrix} =  \begin{pmatrix} 0 & -1 \\ 1 & 0 
\end{pmatrix}\,\begin{pmatrix} x\\ y \end{pmatrix} , \qquad
 f_2\,\begin{pmatrix} x\\ y \end{pmatrix} =\alpha \,  \begin{pmatrix} x\\ y \end{pmatrix} - \alpha\,\begin{pmatrix} 1 \\ 0 \end{pmatrix},\]
and  $0<\alpha <1=t_0$. 
The lower transition attractor $A_*$ is shown
on the left in Figure~\ref{fig:tr1}, with $\alpha =0.4$.  This figure was computed using statement (5) of Theorem~\ref{thm:K}.  The transition hull 
is the square with vertices $(1,0),(0,1), (-1,0), (0,-1)$.    Computer graphics indicate that the upper transition attractor 
$A^* = \lim_{t\rightarrow 1} A_t$ exists and appears as on the right in Figure~\ref{fig:tr1} with $\alpha = 0.3$.  This figure actually shows $A_{0.99}$, computed using the chaos game algorithm with a higher 
probability of choosing the first function to compensate for the value of $t$ close to $1$.   
\end{example}

\begin{example} \label{ex:ta3}  Consider the bounded family
$F_t =\{t f_1(x), t f_2(x)+(1,0) \}$, where
\[f_1\begin{pmatrix} x\\ y \end{pmatrix} =  \begin{pmatrix} 1 & 0 \\ 0 & 1 
\end{pmatrix}\,\begin{pmatrix} x\\ y \end{pmatrix} , \qquad
 f_2\,\begin{pmatrix} x\\ y \end{pmatrix} =\alpha \,   \begin{pmatrix} 0 & -1 \\ 1  & 0 \end{pmatrix}\,\Bigg ( \begin{pmatrix} x\\ y \end{pmatrix} - \begin{pmatrix} 1 \\ 0 \end{pmatrix} \Bigg ),\]
and $0<\alpha <1=t_0$.     The lower transition attractor $A_*$ is
the countable set of points
\[ \Bigg \{\begin{pmatrix} 1 \\ 0 \end{pmatrix} \Bigg \} \bigcup \Bigg \{\alpha^n \begin{pmatrix} 0 & -1 \\ 1  & 0 \end{pmatrix}^n \begin{pmatrix} -1 \\ 0 \end{pmatrix} + \begin{pmatrix} 1 \\ 0 \end{pmatrix} \, :\, n \geq 0 \Bigg\},\]
that spiral around and approach the point $(1,0)$.  The upper transition attractor $A^* = \lim_{t\rightarrow 1} A_t$ appears to exist and is approximated in Figure~\ref{fig:tr3}.
\end{example}

\begin{example} \label{ex:ta2}  Consider  the bounded family
$F_t =\{t f_1(x), t f_2(x)+(1,0) \}$, where
\[f_1\,\begin{pmatrix} x\\ y \end{pmatrix} =\frac{1}{\sqrt{2}}\, \begin{pmatrix} 1 &-1 \\ 1 & 1 \end{pmatrix}\, , \qquad 
f_2\begin{pmatrix} x\\ y \end{pmatrix} =  \alpha \, \frac{1}{\sqrt{2}} \,\begin{pmatrix} 1 & 1 \\ -1 & 1 
\end{pmatrix}\,\Bigg ( \begin{pmatrix} x\\ y \end{pmatrix} - \begin{pmatrix} 1 \\ 0 \end{pmatrix} \Bigg ),\]
 $0<\alpha <1=t_0$.   The lower transition attractor $A_*$ is shown in Figure~\ref{fig:tr2}, with $\alpha =0.4$.   The upper
transition attractor $A^* = \lim_{t\rightarrow 1} A_t$ appears to 
exist and is approximated in Figure~\ref{fig:tr4} with $\alpha = 0.4$.  The figure on the left is $A_{0.9}$; the middle figure is 
$A_{0.94}$; and 
the figure on the right is $A_{0.98}$.  
\end{example}

\begin{figure}[htb]  
\vskip -3mm
\includegraphics[width=7cm, keepaspectratio]{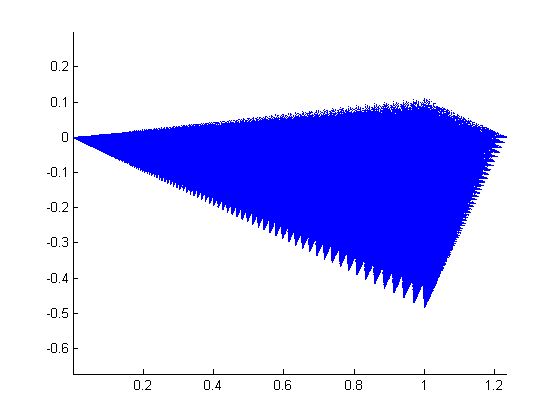}
\caption{The upper transition attractor $A^*$ of Example~\ref{ex:ta3}.}
\label{fig:tr3}
\end{figure}

\begin{figure}[htb]  
\includegraphics[width=8cm, keepaspectratio]{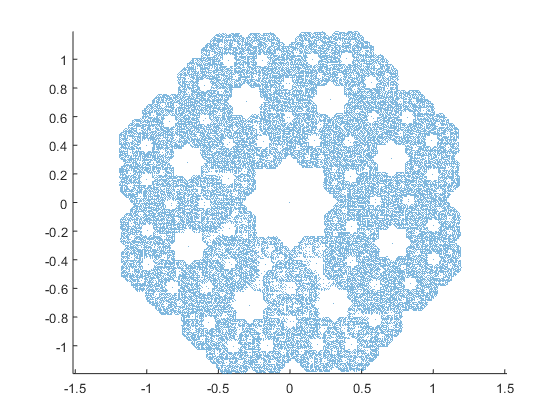} 
\caption{The transition attractor $A_*$ of Example~\ref{ex:ta2}.}
\label{fig:tr2}
\vskip 1cm
\end{figure}

\begin{figure}[htb]  
\includegraphics[width=4cm, keepaspectratio]{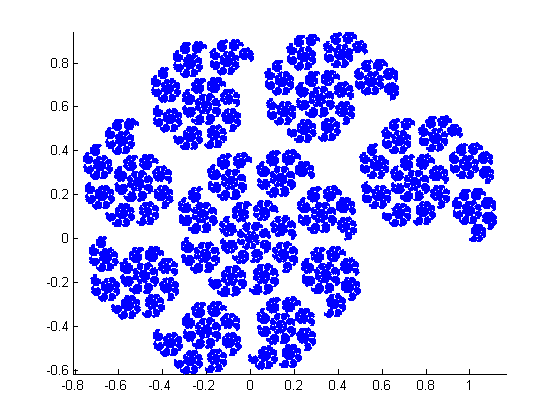} \hskip 4mm \includegraphics[width=4cm, keepaspectratio]{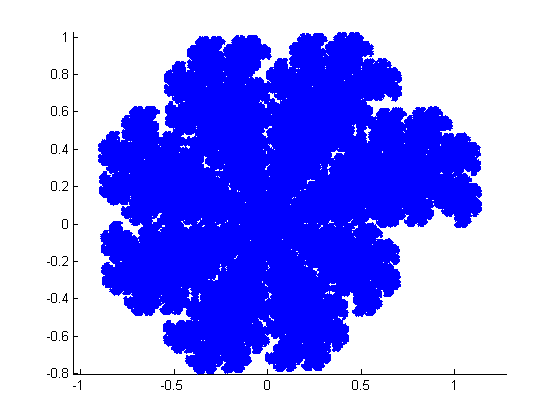} \hskip 4mm \includegraphics[width=4cm, keepaspectratio]{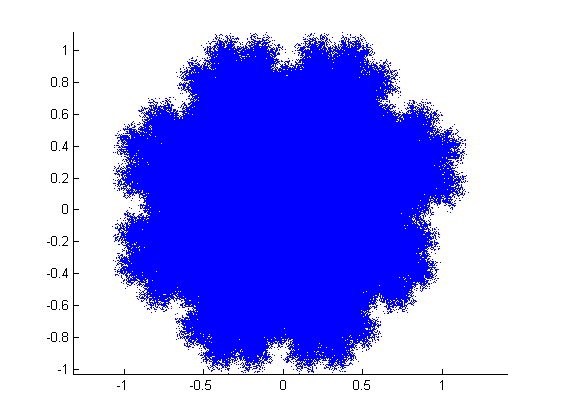}
\caption{The transition attractor $A^*$ of Example~\ref{ex:ta2}.}
\label{fig:tr4}
\vskip 10mm
\end{figure}

\end{document}